\documentclass[11pt,reqno,a4paper]{amsart}

\usepackage{soul}
\usepackage{amsthm}
\usepackage{amssymb}
\usepackage{latexsym}
\usepackage{multicol,multirow}
\usepackage{verbatim,enumerate}
\usepackage{accents,upgreek}
\usepackage[usenames]{color}
\usepackage[colorlinks=true, linkcolor=blue, citecolor=green, urlcolor=blue]{hyperref}
\usepackage[nameinlink,noabbrev,capitalize]{cleveref}
\usepackage{nicematrix,makecell}
\usepackage{amsmath, amscd}
\usepackage{soul,enumitem}
\usepackage{pifont}
\usepackage{dsfont}
\usepackage{dirtytalk,caption,subcaption}
\usepackage{tikz,setspace}
\usetikzlibrary{matrix,arrows,decorations.pathmorphing}
\usetikzlibrary{positioning}

\advance\textwidth by 1.2in \advance\oddsidemargin by -.6in \advance\evensidemargin by -.6in

\theoremstyle{plain}
\newtheorem{prop}{Proposition}
\newtheorem{thm}{Theorem}
\newtheorem{lem}{Lemma}
\newtheorem{cor}{Corollary}

\theoremstyle{definition}
\newtheorem*{example}{Example}
\newtheorem*{defn}{Definition}

\theoremstyle{remark}
\newtheorem{rem}{Remark}

\newcommand{\lie}[1]{\mathfrak{#1}}

\newcommand{\mr}[1]{\mathring{#1}}
\newcommand\br{\mathbb R}
\newcommand\bc{\mathbb C}
\newcommand\bz{\mathbb Z}

\newcommand{\xmark}{\ding{55}}%

\def\a{\alpha}

\newcounter{cnt}
 \makeatletter
\def\mydggeometry{\makeatletter\dg@YGRID=1\dg@XGRID=20\unitlength=0.003pt\makeatother}
\makeatother \theoremstyle{remark}

\numberwithin{equation}{section}
\makeatletter
\def\section{\def\@secnumfont{\mdseries}\@startsection{section}{1}%
	\z@{.7\linespacing\@plus\linespacing}{.5\linespacing}%
	{\normalfont\scshape\centering}}
\def\subsection{\def\@secnumfont{\bfseries}\@startsection{subsection}{2}%
	{\parindent}{.5\linespacing\@plus.7\linespacing}{-.5em}%
	{\normalfont\bfseries}}
\makeatother

\title{Maximal root subsystems of affine reflection systems and duality}

\author{Irfan Habib}\address{The Institute Of Mathematical Science, A CI Of Homi Bhabha National Institute, Chennai 600113, India}
\email{irfanhabib@imsc.res.in}
\thanks{}

\begin{document}
	\begin{abstract}
		Any maximal root subsystem of a finite crystallographic reduced root system is either a closed root subsystem or its dual is a closed root subsystem in the dual root system. In this article, we classify the maximal root subsystems of an affine reflection system (reduced and non-reduced) and prove that this result holds in much more generality for reduced affine reflection systems. Moreover, we explicitly determine when a maximal root subsystem is a maximal closed root subsystem. Using our classification, at the end, we characterize the maximal root systems of affine reflection systems with nullity less than or equal to $2$ using Hermite normal forms; especially for Saito's EARS of nullity $2.$ This in turn classifies the maximal subgroups of the Weyl group of an affine reflection system that are generated by reflections.
	\end{abstract}
	
	\maketitle
	
	\section{Introduction}
	
	In \cite{hoegh1990classification} H{\o}egh--Krohn and Torresani introduced a new class of infinite-dimensional Lie algebras which generalizes both finite-dimensional simple Lie algebras and affine Lie algebras. They called the Lie algebras as quasi-simple Lie algebras. A \textit{quasi-simple Lie algebra} $\mathcal{L}$ is equipped with an invariant, symmetric, non-degenerate bilinear form which has a distinguished finite-dimensional self-centralizing subalgebra $\mathcal{H}.$ Moreover, the Lie algebra $\mathcal{L}$ has a weight space decomposition with respect to $\mathcal{H}.$ Since then, a vast number of research has been done to understand them, specially after \cite{allison1997extended} where the name Extended Affine Lie algebra (EALA in short) appeared for the first time. One other standard reference for EALA is \cite{neher2004extended}.
	
	More generally, Lie algebras $\mathcal{L}$ having weight space decomposition with respect to a non-zero abelian subalgebra $\mathcal{H}$ (called \textit{toral subalgebra}) form a large class of Lie algebras. We can assign to each such Lie algebra $\mathcal{L}$ a subset $\Delta$ of the dual space $\mathcal{H}^*$ of $\mathcal{H},$ called its \textit{root system}. 
	The interaction of the Lie algebra with its root system provides an approach to studying the structure of Lie algebra via its root systems. For a finite-dimensional simple Lie algebra $\mathcal{L}$, the root system of $\mathcal{L}$ completely determines $\mathcal{L}$ upto isomorphism (see \cite{humphreys2012introduction}). There is a non-degenerate  bilinear form on $\mathcal{H}^*$ induced by the non-degenerate bilinear form on $\mathcal{H}$. A non-isotropic root relative to this form is called a \textit{real root}. Assume that $\Phi$ is the collection of all real roots in $\Delta.$  There is an associated group, called the \textit{Weyl group}, which is a subgroup of $GL(\mathcal{H}^*)$ (the group of invertible operators on $\mathcal{H}^*$) and is by definition generated by the reflections $\{s_\alpha: \alpha\in \Phi\}$. A subset $\Psi$ of $\Phi$ is called a \textit{root subsystem} if $s_\alpha(\beta)\in\Psi$ for all $\alpha,\beta\in\Psi.$ The reflection subgroups of the Weyl group are in one-to-one correspondence with the root subsystems of $\Phi.$ In particular, maximal root subsystems correspond to maximal reflection subgroups of the Weyl group. This in turn provokes a systematic study of the structure of the root systems apart from its connection with the Lie algebra (see \cite{loos2004locally,neher2011extended,yoshii2004root}). One of the motivations of this work is to study the reflection subgroups of the Weyl groups of affine reflection systems (which contains studying the reflection subgroups of extended affine Lie algebras) with more emphasis on the combinatorial aspects.
	
	The classification of the reflection subgroups of finite and affine Weyl groups has been achieved in \cite{dyer2011reflection} by classifying 
	root subsystems of finite and real affine root systems. 
	It is well-known that the classification of all root subsystems may be deduced from that of the closed root subsystems in the finite setting, see \cite{Carter72}. The root subsystems of untwisted affine Lie algebras have been characterized in \cite{dyer2011root}. The maximal root subsystems for twisted affine root systems were still un-characterized and in this article, we shall describe them as Corollary of the classification Theorem.
	
	A subset $\Psi$ of $\Phi$ is called a \textit{closed subset} if $\alpha+\beta\in\Psi$ whenever $\alpha,\beta\in\Psi$ and $\alpha+\beta\in\Phi.$ If in addition, $\Psi$ is a root subsystem, we call $\Psi$ a \textit{closed root subsystem} .  Closed root subsystems (subsets) have been studied extensively since they are closely related to the $\mathcal{H}$ invariant subalgebras of $\mathcal{L}.$ The problem of classifying closed subsets even for (real) affine root systems is wide open. 
	The invertible closed subsets of affine root systems were classified in \cite{CKS98} and the parabolic subsets were classified in \cite{Futorny}. 
	Anna Felikson et al. started the classification of regular subalgebras of affine Lie algebras (which are related to the closed root subsystems of real affine root systems) in \cite{FRT08} and it was completed in \cite{roy2019maximal}, see also \cite{kus2021borel}. A combinatorial description of biclosed sets of real affine root systems was given in \cite{BS22}.
	In the finite-dimensional setting, closed subsets of a root system are in one-to-one correspondence with the Cartan invariant subalgebra (see \cite[Proposition 4.1]{douglas2021closed}). Maximal closed root subsystems were classified in \cite{BdS} in the finite-dimensional setting, in \cite{roy2019maximal} for affine root systems and in \cite{kus2021borel} for affine reflection systems. Reflectable bases for affine reflection systems were studied in \cite{azam12reftectable}. Symmetric closed subsets of affine root systems and their correspondence with the regular subalgebras had been studied in \cite{biswas2023symmetric}. For a general symmetrizable Kac--Moody algebra, in \cite{habib2023root}, the authors proved that Dynkin's results hold and they identified the correct family of subalgebra which are determined by the combinatorics of their root systems. The embedding problem for rank $2$ Kac--Moody Lie algebras was  addressed combinatorially in \cite{habib2024pi}.
	
	Affine reflection systems generalize all the above-mentioned root systems (see Section \ref{secdefars} for a precise definition). It includes finite root systems, affine root systems, extended affine root systems to name a few. In this work, we will classify the maximal root subsystems of an affine reflection system, which, in turn, will provide a more explicit classification of the maximal subgroups of its Weyl group generated by reflections. One other motivation of this work is the following result in the finite-dimensional setting: If $\Psi$ is a maximal root subsystem of a finite reduced root system $\Phi,$ then either $\Psi$ is closed in $\Phi$ or $\Psi^\vee$ is closed in $\Phi^\vee$ (see \cite[Corollary 2]{dyer2011reflection}). We provide a positive answer to this statement for reduced affine reflection systems in Theorem \ref{thmclsdordualclsd}. The proof based on the classification Theorems; Theorem \ref{thmrednotoftypeB} and Theorem \ref{thmredoftypeB}, where we classify 
	the maximal root subsystems of an irreducible reduced affine reflection system depending on the gradient type whereas Theorem \ref{thmnonred} classifies the same for non-redcued affine reflection systems. Moreover, we also explicitly describe when a maximal root subsystem of an affine reflection system is closed in Table \ref{Tablereduced} for reduced and Remark \ref{nonredmaxclsd} for non-reduced.
	In the last Section, we apply our classification Theorems for affine reflection systems of nullity at most $2.$ We explicitly classify the maximal root subsystems of Saito's EARS defined in \cite[Section 5]{saito1985extended} using the Hermite normal forms.

	\textit{The paper is organized as follows:} In Section \ref{section2} we recall facts about finite root systems, construction of affine reflection systems and lattices in finite-dimensional real vector spaces along with the related Hermite normal forms. We also introduce the notion of dual motivated by \cite{macdonald2003affine}. In Section \ref{secrootsubsystems} we review the basic facts about the root subsystems and the dual of them. Section \ref{secrednotB} and Section \ref{secredB} classify the maximal root subsystems of a reduced affine reflection system where the gradient is not of type $B$ and type $B$ respectively. Section \ref{secnonred} is devoted to the classification of the maximal root subsystems of a non-reduced affine reflection system. Finally, Section \ref{sectionnullityle2} provides applications of the classification results for affine reflection systems of nullity at most $2$.
	
	
	\section{Extension datum and affine reflection systems}\label{section2}
	Throughout this paper we denote by $\mathbb{C}$ (resp. $\mathbb{R}$) the field of complex (resp. real) numbers and by $\mathbb{Z}$ (resp. $\mathbb{Z}_{+}$, $\mathbb{N}$) the subset of integers (resp. non-negative, positive integers). For a subset $A$ of a group $G,$ we denote by $\langle A\rangle$ the subgroup of $G$ generated by $A.$
	
	\subsection{Finite root systems} Let $I$ be a finite index set often identified with $\{1,2,\dots,n\}$ if $|I|=n.$ Let $V$ be a Euclidean space with an inner product $(\cdot,\cdot)$. For the rest of this paper we denote by $\mathring{\Phi}$ a finite root system in $V$. For $\alpha\in\mr{\Phi}$, we can define a reflection $s_\alpha:V\to V$ given by $s_\alpha(v):=v-2\frac{(v,\alpha)}{(\alpha,\alpha)}\alpha,\ v\in V.$
	The \textit{Weyl group} $\mr{W}$ of $\mr{\Phi}$ is defined to be the subgroup of the invertible linear maps on $V$ generated by the reflections $s_{\alpha},\alpha\in\mathring{\Phi}.$ A root system $\mr{\Phi}$ is called \textit{reduced} if $\alpha/2\notin \mr{\Phi}$ for all $\alpha\in \mr{\Phi}$ and called \textit{reducible} if there exist disjoint subsets $\mr{\Phi}_1,\mr{\Phi}_2$ of $\mr{\Phi}$ such that $\mr{\Phi}=\mr{\Phi}_1\cup \mr{\Phi}_2$ and $(\mr{\Phi}_1,\mr{\Phi}_2)=0.$ A root system that is not reduced (resp. reducible) is called \textit{non-reduced} (resp. \textit{irreducible}). If $\mathring{\Phi}$ is reduced and irreducible, then at most two root lengths occur in $\mathring{\Phi}$ and all roots of a given length are conjugate under $\mathring{W}$. We denote the short (resp. long) roots in $\mathring{\Phi}$ by $\mr{\Phi}_s$ (resp. $\mr{\Phi}_\ell$). If $\mathring{\Phi}$ is non-reduced and irreducible, then $\mr{\Phi}$ is of type $BC_I$ and we set
	$$\mathring{\Phi}_s=\{\pm\epsilon_i,i\in I\},\ \ \mathring{\Phi}_\ell=\{\pm\epsilon_i\pm\epsilon_j: i\neq j,\ i,j\in I\},\ \ \mathring{\Phi}_d=\{\pm2\epsilon_i,i\in I\},$$ where $\{\epsilon_i:i\in I\}$ is the standard orthonormal basis of $V.$ We define $\mr{\Phi}_{nd}:=\mr{\Phi}\backslash \mr{\Phi}_d.$ For an irreducible root system $\mr{\Phi},$ we denote by $m_{\mr{\Phi}}$ or simply by $m$ the lacing number of $\mr{\Phi}.$ For more details about finite root systems, see \cite{bourbaki2002lie,humphreys2012introduction}. Reduced root systems appear as the set of roots of finite-dimensional semisimple Lie algebras with respect to a Cartan subalgebra \cite{humphreys2012introduction} and non-reduced root systems appear in the context of infinite-dimensional Lie algebras \cite{kac1990infinite}. One can also define locally finite root systems e.g. \cite{loos2004locally} since they appear naturally in the theory of affine reflection systems, see Theorem \ref{strthm}.

	\subsection{Affine reflection systems}\label{secdefars} Now we introduce the main object of study, the affine reflection systems. For more details about affine reflection systems, the reader is referred to \cite{loos2011reflection,neherextendedaffine2011}.
	Let $X$ be a finite-dimensional real vector space, $(\cdot,\cdot)_X$ a symmetric bilinear form on $X$ and $\Delta\subseteq X$. Define 
	\begin{align*}X^0=\{x\in X: (x,X)_X=0\},\ \ & \ \ \Delta^{\text{im}}=\{\alpha\in \Delta: (\alpha,\alpha)_X=0\} ,\ (\text{imaginary roots})\\ \Delta^{\text{re}} =\{\alpha\in \Delta:& (\alpha,\alpha)_X\neq 0\},\ (\text{real roots}).\end{align*}
	For $\alpha\in\Delta^{\text{re}}$ and $x\in X,$ we define $(x,\alpha^\vee)=2\frac{(x,\alpha)_X}{(\alpha,\alpha)_X}$ and $s_\alpha(x)=x-(x,\alpha^\vee)_X\alpha.$
	\begin{defn}
		We call a triple $(X,\Delta,(\cdot,\cdot)_X)$ an \textit{affine reflection system} if the following hold:
		\begin{enumerate} 
			\item [(i)] $0\in \Delta,\  \Delta \text{ spans $X$},\  \Delta^{\text{im}}\subseteq X^0$,
			\item  [(ii)] $s_{\alpha}(\Delta)=\Delta,\ \ \forall \alpha\in \Delta^{\text{re}},$
			\item  [(iii)] $(\Delta,\alpha^{\vee})_X\subseteq\bz$ is a finite subset $\forall \alpha\in \Delta^{\text{re}}$.
		\end{enumerate}
		
		The \textit{Weyl group} $\mathcal{W}$ of $(X,\Delta,(\cdot,\cdot)_X)$ is defined to be the subgroup of $\mathrm{GL}(X^*)$ generated by the reflections $s_\alpha,\alpha\in\Delta^{\mathrm{re}}.$ We set $\Phi:=\Delta^{\text{re}}$ for the rest of the paper. We shall sometimes ignore the imaginary roots and write  $(X,\Phi,(\cdot,\cdot)_X)$ or simply $\Phi$ to denote an affine reflection system. Moreover, we shall write $(\cdot,\cdot)$ instead of $(\cdot,\cdot)_X$ when the underlying vector space $X$ is clear from the context.

	\end{defn}
	
	\subsection{Extension datum} It turns out that any reflection system can be constructed from a root system and an extension datum.
	\begin{defn}\label{extda} Let $\mathring{\Phi}$ be a locally finite root system in $V$ and $Y$ be a finite-dimensional real vector space. We call a collection $(\Lambda_{\alpha}: \alpha\in \mathring{\Phi}\cup\{0\})$, $\Lambda_{\alpha}\subseteq Y$ an extension datum of type $(\mathring{\Phi},Y)$ if it satisfies the following axioms:
		\begin{enumerate}
			\item $\Lambda_{\beta}-(\beta,\a^{\vee})\Lambda_{\a}\subseteq \Lambda_{s_{\a}(\beta)},\ \ \forall \alpha,\beta\in \mathring{\Phi}$.
			\item $0\in\Lambda_{\alpha}$ for $\alpha\in \mathring{\Phi}_{nd}\cup\{0\}$ and $\Lambda_{\alpha}\neq \emptyset$ for $\alpha\in \mathring{\Phi}_{d}$.
			\item $Y=\text{span}_{\mathbb{R}}(\bigcup_{\alpha\in \mathring{\Phi}\cup\{0\}} \Lambda_\alpha)$.
		\end{enumerate}
	\end{defn}
	
	Given an extension datum of type $(\mathring{\Phi},Y),$ we can construct an affine reflection system $(X,\Delta,(-,-)_X)$ as follows: 
	\begin{equation}\label{consaff}X=V\oplus Y,\ \ \Delta=\bigcup_{\alpha\in \mathring{\Phi}\cup\{0\}}\alpha\oplus \Lambda_{\alpha},\ \ (v_1\oplus y_1, v_2\oplus y_2)_X=(v_1,v_2).\end{equation}

	The following result is the structure Theorem for affine reflection systems \cite[Theorem 4.6]{loos2011reflection}.
	\begin{thm}\label{strthm}
		Let $\mathring{\Phi}$ be a locally finite root system and $(\Lambda_{\alpha}: \alpha\in \mathring{\Phi}\cup\{0\})$ be an extension datum of type $(\mathring{\Phi},Y)$. The triple $(X,\Delta,(\cdot,\cdot)_X)$ constructed in Equation \eqref{consaff} is an affine reflection system with 
		$$\Delta^{\text{im}}=\Lambda_0,\ \ X^0=Y,\ \ \Delta^{\text{re}}=\bigcup_{\alpha\in \mathring{\Phi}}\alpha\oplus \Lambda_{\alpha}.$$
		Moreover, any affine reflection system is isomorphic to an affine reflection system constructed in this way and $(X,\Delta,(\cdot,\cdot)_X)$ is irreducible if and only if the underlying finite root system $\mathring{\Phi}$ is irreducible.
		\hfill\qed
	\end{thm}
	
	Although most of the results of this paper are valid when we assume that $\mathring{\Phi}$ is a locally finite root system, to avoid technicalities, \textit{we shall always assume that if $(X,\Delta,(\cdot,\cdot)_X)$ is an affine reflection system given by Theorem \ref{strthm}, then $\mathring{\Phi}$ is a finite root system.} We call an affine reflection system $(X,\Delta,(\cdot,\cdot)_X)$ \textit{reduced} (resp. \textit{non-reduced}) if the underlying finite root system $\mathring{\Phi}$ is reduced (resp. non-reduced).
	\begin{example}
		\begin{enumerate}
			\item [(i)] If $\mr{\Phi}$ is a finite root system, then $\mathring{\Phi}\cup\{0\}$  is an affine reflection system of nullity $0$. Conversely, an affine reflection system of nullity $0$ if and only if it is a finite root system union $\{0\}\}$.
			\item [(ii)] The affine root system is an affine reflection system of nullity $1.$ Conversely, an affine reflection system of nullity $1$ if and only if it is an affine root system (see e.g. \cite{neher2011extended}).
			\item [(iii)] The set of roots of (untwisted) toroidal Lie algebra $\lie g\otimes \bc[t_1^{\pm 1},t_2^{\pm 1},\dots,t_k^{\pm 1}]$ is an affine reflection system of nullity $k.$
			\item [(iv)] More generally, the roots of an extended affine Lie algebra also form an affine reflection system (\cite[Section 3.4]{neherextendedaffine2011}).  
		\end{enumerate}
	\end{example}

	\subsection{Properties of extension datum}\label{propexdat} We record some elementary properties of extension datum. The next Proposition can be found in \cite[Exercise 3.16]{neherextendedaffine2011} and \cite[Theorem 3.18]{neherextendedaffine2011}.
	\begin{prop}\label{propsametyperoots}
		Let $(\Lambda_{\alpha},\alpha\in\mathring{\Phi}\cup \{0\})$ be an extension datum of type $(\mathring{\Phi},Y)$. Then for all $\alpha\in\mathring{\Phi}$ we have
		$$\Lambda_{\alpha}=\Lambda_{w(\alpha)},\ \forall w\in \mr{W},\ \ \Lambda_{\alpha}=\Lambda_{-\alpha}=-\Lambda_{\alpha}.$$
	\end{prop}
	In particular, if $\mathring{\Phi}$ is irreducible, 
	then an extension datum of type $(\mathring{\Phi},Y)$ consists of at most 4 different subsets, namely $\Lambda_{0},\Lambda_{s},\Lambda_{\ell},\Lambda_{d}$ defined in the obvious way. 
	By Theorem \ref{strthm} there exists an irreducible root system $\mathring{\Phi}$ and an extension datum $(\Lambda_{0},\Lambda_{s},\Lambda_{\ell},\Lambda_{d})$ of type $(\mathring{\Phi},Y)$ such that (up to isomorphism)
	\begin{equation*}\Phi=(0\oplus \Lambda_0) \cup \bigcup_{\alpha\in \mathring{\Phi}_s} (\alpha\oplus \Lambda_s) \cup  \bigcup_{\alpha\in \mathring{\Phi}_\ell} (\alpha\oplus \Lambda_\ell)\cup  \bigcup_{\alpha\in \mathring{\Phi}_d} (\alpha\oplus \Lambda_d)\end{equation*}
	From Definition \ref{extda}(1), the subsets $\Lambda_s,\Lambda_{\ell},\Lambda_d$ satisfy:\label{lambdainclusions}
	$$\Lambda_x+2\Lambda_x\subseteq \Lambda_x,\ \ \ x\in \{s,\ell\},\ \ \ \Lambda_s+\Lambda_\ell\subseteq \Lambda_s,\ \ \ \Lambda_\ell+m\Lambda_s\subseteq \Lambda_\ell,$$
	$$\Lambda_s+\Lambda_d\subseteq \Lambda_s,\ \ \ \Lambda_d+4\Lambda_s\subseteq \Lambda_d,\ \ \ \Lambda_\ell+\Lambda_d\subseteq \Lambda_\ell,\ \ \ \Lambda_d+2\Lambda_\ell\subseteq \Lambda_d,$$
	$$\Lambda_x+\Lambda_x\subseteq \Lambda_x \text{ if }\mr{\Phi}\text{ contains a triple $(\alpha,\beta,\alpha+\beta)$ of type }(x,x,x),\ \ x\in \{s,\ell\}.$$

	\subsection{Duality}\label{sectionduality} We now define the duality of affine reflection systems for real roots. This definition is motivated by the definition of dual in \cite[Section 1.3]{macdonald2003affine} for affine root systems. Given an affine reflection system  $(X,\Delta^{\text{re}},(\cdot,\cdot))$ and $\alpha\in \Delta^{\mathrm{re}},$ we define  the dual of $\alpha$ by $\alpha^\vee:=\frac{2\alpha}{(\alpha,\alpha)}.$ 
	
	\begin{defn}\label{defdualars}
		Let $\mathcal{A}=(X,\Delta^{\text{re}},(-,-))$ be a reduced affine reflection system. We define the dual of $\mathcal{A}$ by $\mathcal{A}^\vee:=(X^\vee,(\Delta^\vee)^{\text{re}},(-,-)^\vee)$ where $$X^\vee:=X,\quad (-,-)^\vee=(-,-),\quad \text{ and }(\Delta^\vee)^{\text{re}}:=\{\alpha^\vee:\alpha\in\Delta^{\text{re}}\}.$$ 
	\end{defn}
	\begin{rem}
		Since the radical of the form $(-,-)$ on $X$ is $X^0,$ the induced form on the quotient $X/X^0$ is non-degenerate. This in turn induces a non-degenerate form on the dual $(X/X^0)^*$ via the linear isomorphism $$\nu:X/X^0\to (X/X^0)^*, \quad x\mapsto \nu(x) \text{ and } \nu(x)(y):=(x,y)\ \ \forall x,y\in X/X^0$$ We can define a bilinear form $(-,-)^*$ on $X^*$ induced by the form on $(X/X^0)^*$ and declaring its radical to be $(X^0)^*.$ For a real root $\alpha\in \Phi$ we can define  $\alpha^*$ to be the unique element of $X^*$ such that $$\alpha^*(\alpha+X^0)=2,\ \ \alpha^*(X^0)=0.$$  Consequently, we can also define the dual $\mathcal{A}^\vee$ of $\mathcal{A}$ by $$X^\vee:=X^*,\quad (-,-)^\vee=(-,-)^*,\quad \text{ and }\Delta^\vee:=\{\alpha^*:\alpha\in\Delta\}.$$
		
		If $\alpha$ is real, then it is easy to see that we can identify $\alpha^*$ as the functional $x\mapsto 2\frac{(x,\alpha)}{(\alpha,\alpha)}.$ Hence \textit{we shall always assume the dual of an affine reflection system is defined by Definition \ref{defdualars}}. Note that as in the affine case, this definition does not extend to the non-reduced case since $\Lambda_d$ may not contain zero but it is necessary that $0\in \Lambda_s^\vee=\Lambda_d/2$ (see definition \ref{extda}).
	\end{rem}
	\begin{example}
		Let $\Phi$ be the affine reflection system of type $C_I$ and nullity $2$ defined by $$\Lambda_s=\bz\times \bz,\ \ \Lambda_\ell=\{(0,0),(1,0),(0,1)\}+2\bz\times 2\bz.$$
		Then, we have $\Lambda_s^\vee=(1/2)\Lambda_\ell$ and $\Lambda_\ell^\vee=\Lambda_s.$  After scaling, we can assume that $\Lambda_s^\vee=\Lambda_\ell$ and $\Lambda_\ell^\vee=2\Lambda_s.$ Note that the condition $\langle \Lambda_\ell\rangle=\Lambda_s$ becomes $2\langle \Lambda_s^\vee\rangle=\Lambda_\ell^\vee$ in the dual.
	\end{example}

	\subsection{Hermite normal form and lattices} A \textit{lattice} $L$ in a finite-dimensional real vector space $V$ is a free $\bz$ submodule of $V.$ The \textit{rank} of a lattice $r(L)$ is defined to be $r(L):=\dim_{\br}( \br\otimes_{\bz}L).$ A \textit{full lattice} is a lattice $L$ with $r(L)=\dim(V).$ If $\dim V=n,$ any lattice $L$ in $V$ has a \textit{Hermite normal basis} $\mathcal{B}=\{v_1,v_2,\dots,v_r\}$ such that the $r\times n$ matrix with rows $v_1,v_2,\dots,v_r$ is an upper triangular integer matrix and satisfies the following properties:
	\begin{enumerate}
		\item The leading coefficient (also called pivot) of a non-zero row is always strictly to the right of the leading coefficient of the row above,
		\item The elements below the pivot are all zero and the elements above the pivot are all non-negative and strictly less that the pivot.
	\end{enumerate}
	
	For more about lattices and related Hermite normal form, we refer to \cite{cohen1993course,micciancio2002complexity}. We only recall the facts that are necessary for our purpose. The following can be proved using the Gram--Schmidt orthogonalization; see also \cite[Chapter 1]{martinet2013perfect} and \cite{zong2021classification}.
	\begin{prop}\label{det}
		Let $L$ be a lattice in $\br^n$ of rank $r$ and $B=\{v_1,v_2,\dots,v_r\}$ be a basis of $L.$ Let $M_B$ be the $r\times n$ matrix whose $i$-th row is $v_i$ $i=1,2,\dots,r.$ Then $$\prod_{i=1}^r\|v_i^*\|=\sqrt{\det(M_B(M_B)^t)}\ ,$$ where 
		$\{v_1^*,v_2^*,\dots,v_r^*\}$ is the orthogonal basis of the $\br$-subspace spanned by $B$ obtained by applying the Gram--Schmidt orthogonalization to $B.$
	\end{prop}
	Since any two bases of a lattice are related by a unimodular matrix, the right hand side of the above equality is independent of the basis chosen. We denote by $d(L)$ the determinant of $L$ which is the common value given by Proposition \ref{det}. Note that $d(L)=|\det(M_B)|$ if $L$ is a full rank lattice. The next Theorem holds in much more generality (see \cite[Chapter 12]{dummit2004abstract}). 
	
	\begin{thm}\label{thmfgabelian}
		Let $A$ be a free abelian group of finite rank $n$ and $B\neq 0$ be a subgroup of $A$. Then there exists a $\bz$-basis $\{a_1,a_2\dots,a_n\}$ of $A$, and integers $d_1,\dots,d_r$ with $1\le r\le n, 1\le d_1\mid d_2\mid \dots \mid d_r,$ such that $\{d_1a_1,d_2a_2,\dots ,d_ra_r\}$ is a $\bz$-basis for $B$. In particular, $B$ is a free abelian group of rank $r\le n$.
	\end{thm}
	\begin{cor}
		If $L_1\subseteq L_2$ are lattices in $\br^n$, then $d(L_2)\mid d(L_1).$
	\end{cor}
	\begin{proof}
		Follows from Proposition \ref{det} and Theorem \ref{thmfgabelian}.
	\end{proof}
	\begin{rem}
		Let $L$ be a lattice and $L'$ be a full rank sublattice of $L.$ Then $L=L'$ if and only if $d(L)=d(L').$ Moreover, two different sublattices of same determinant are incomparable with respect to $``\subseteq"$ relation.
	\end{rem}

	\noindent 
	We end this Section with a couple of Lemmas. The proof of the first Lemma is easy and we skip the details. The second Lemma plays a crucial role in the main Theorem. For a subgroup $H$ of a group $G,$ if $A$ is a union of cosets of $H$ in $G,$ we say $S$ is a union of cosets in $A/H$ to mean that $S$ is a union of cosets of $H$ in $G$ such that $S\subseteq A$.
	
	\begin{lem}\label{lemmaximalsubgroup}
		Let $L$ be a lattice in a finite-dimensional real vector space and $M$ be a maximal sublattice of $L$ so that $d(M)/d(L)=p$ for some prime $p.$ Then $qL\subseteq M$ for some prime $q$ if and only if $p=q.$
	\end{lem}

	\begin{lem}\label{keylemmalattice}
		Let $\Lambda$ be a lattice in a finite-dimensional real vector space $V.$ Assume that $$L:=\bigcup_{i=0}^r(a_i+2\Lambda),\ \ \ L':=\bigcup_{i=0}^k(b_i+2\Lambda),\ \ \ \ a_0=b_0=0$$ are two union of cosets in $\Lambda/2\Lambda.$ Let $H$ be a maximal sublattice of $\Lambda$  such that $d(H)/d(\Lambda)\neq 2$ and $H\cap L\subseteq L'.$ Then $L\subseteq L'$ and $k\geq r.$
	\end{lem}
	
	\begin{proof}
		Let $d(H)/d(\Lambda)=p.$ For simplicity we identify $\Lambda$ with $\bz^n$ and denote the standard basis vectors of $\bz^n$ by $\{e_1,e_2,\dots, e_n\}.$ Let $\{v_1,v_2,\dots, v_n\}$ be the Hermite normal basis of $H.$ Then there exist $k$ and non-negative integers $x_1,\dots,x_{k-1}$ satisfying $x_t<p$ for each $t$ and 
		$$v_j=\begin{cases}
			e_j+x_je_k &\text{ if } 1\le j< k,\\
			pe_k &\text{ if } j=k,\\
			e_j &\text{ if } k<j\le n.
		\end{cases}$$ 
		We shall show that $L\subseteq L'$ by showing that $H\cap L$ is a union of $r$ distinct cosets modulo $2\Lambda.$ It is clear that $$H\cap L=\bigcup_{i=0}^rH\cap(a_i+2\Lambda).$$ Fix arbitrary $i\in\{0,1,\cdots,r\}$ and let $a_i=(a_{ij})_j.$ Since $p\neq 2$, for each $j\le k$, by Chinese remainder Theorem there exists a unique solution (modulo $2p\bz$) of the system of equations given by $$x\equiv 0\ (\mathrm{mod}\ p),\ \ x\equiv a_{ij}\ (\mathrm{mod}\ 2).$$
		We call this solution $x_{ij}$ and define $a_i':=(a_{ij}')_j$ where $a_{ij}'=x_{ij}$ if $j\le k$ and $a_{ij}$ otherwise. Note that we have
		$$a_i'=\sum\limits_{\substack{j=1\\ j\neq k}}^n a_{ij}'v_j+\left(\frac{x_{ik}-\sum_{j=1}^{k-1}x_{ij}x_j}{p}\right)v_k\in H.$$
		Moreover, since $a_i'-a_i\in 2\Lambda$ it follows that 
		$(a_i'+H\cap 2\Lambda)\subseteq H\cap L\subseteq L'$ for $i=0,1,\dots,r$ and
		therefore $a_i'\in L'$ which implies $a_i'+2\Lambda\in L'.$ Consequently,
		$$L=\bigcup_{i=0}^r(a_i+2\Lambda)=\bigcup_{i=0}^r(a_i'+2\Lambda)\subseteq L'.$$
		This completes the proof.
	\end{proof}

	\section{Root subsystems of affine reflection systems}\label{secrootsubsystems}
	In this section we recall the definitions and elementary properties of root subsystems. Throughout this article, by root subsystem we shall always mean real root subsystem.
	\subsection{} Let $\Phi$ an affine reflection system. A non-empty subset $\Psi\subseteq \Phi$ is called a \textit{root subsystem} if $s_{\alpha}(\beta)\in \Psi$ for all $\alpha,\beta\in\Psi$. A root subsystem $\Psi$ is called a \textit{maximal root subsystem} if $\Psi\subseteq \Psi'\subsetneq \Phi$ implies $\Psi=\Psi'$ for all root subsystems $\Psi'$ of $\Phi.$ We call a root subsystem \textit{closed} if $\alpha,\beta\in \Psi$ and $\alpha+\beta\in \Phi$ implies $\alpha+\beta\in \Psi.$
	Given a root subsystem $\Psi,$ define 
	$$\mathrm{Gr}(\Psi):=\{\alpha \in  \mathring{\Phi}: \exists \ y\in\Lambda_{\a} \text{ such that } \alpha\oplus y\in \Psi\},$$ $$Y_\alpha(\Psi):=\{y\in \Lambda_{\alpha}: \alpha\oplus y\in \Psi\},\ \alpha\in \mathrm{Gr}(\Psi).$$
	
	
	\noindent Moreover, for a root subsystem $\mr{\Psi}$ of $\mr{\Phi}$ we define $\widehat{\mr{\Psi}}:=\{\alpha\oplus \Lambda_\alpha:\alpha\in\mr{\Psi}\}.$  Note that $\mr{\Psi}$ is closed in $\mr{\Phi}$ if and only if $\widehat{\mr{\Psi}}$ is closed in $\Phi.$ The proof of the following Lemma can be found in \cite{kus2021borel}.
	\begin{lem}\label{lemm1}Let $\Phi$ an affine reflection system and $\Psi$ be a root subsystem of $\Phi$. Then $\mathrm{Gr}(\Psi)\subseteq \mathring{\Phi}$ is a root subsystem of $\mathring{\Phi}\cup \{0\}$. Moreover, we have
		$$Y_\beta(\Psi)-(\beta,\alpha^{\vee})Y_\alpha(\Psi)\subseteq Y_{s_{\alpha}(\beta)}(\Psi),\ \ \forall \alpha,\beta\in \mathrm{Gr}(\Psi).\qed$$
	\end{lem}
	Given a root subsystem $\Psi,$ we define $W_\Psi$ be the subgroup of $\mathcal{W}$ generated by reflections $s_\alpha,\alpha\in\Psi.$ The map $\Psi\mapsto \mathcal{W}_\Psi$ is a bijection between the root subsystems of $\Phi$ and the subgroups of $\mathcal{W}$ generated by reflections, the inverse map being $\mathcal{W}'\mapsto\{\alpha\in \Phi:s_\alpha\in\mathcal{W}'\}$.

	\subsection{}\label{sectiondefinitionp} If $\mr{\Phi}$ is reduced, we introduce a $\bz$--linear function 
	$$p: \mathrm{Gr}(\Psi)\rightarrow \bigcup_{\a\in \mathring{\Phi}}\Lambda_{\a},\ \alpha\mapsto p_{\alpha}$$
	as follows. By Lemma \ref{lemm1} and \cite[Theorem 10.1]{humphreys2012introduction} we can choose a simple system $\Pi\subseteq \mathrm{Gr}(\Psi)$ and arbitrary $p_{\gamma}\in Y_{\gamma}(\Psi)$ for each $\gamma\in \Pi$. We extend this $\bz$--linearly and obtain
	\begin{equation}\label{keyequationp}
		p_{\beta}-(\beta,\alpha^{\vee})p_{\alpha}=p_{s_{\alpha}(\beta)}.
	\end{equation} 
	Now it is easy to see with Equation \eqref{keyequationp} that $p_{\alpha}\in Y_\alpha(\Psi)$ for all 
	$\alpha\in \mathrm{Gr}(\Psi)$. We set $Y_\alpha'(\Psi)=Y_\alpha(\Psi)-p_{\alpha}$ for $\alpha\in \mathrm{Gr}(\Psi)$ and observe that $0\in Y_\alpha'(\Psi)$ for all $\alpha\in \mathrm{Gr}(\Psi)$. 
	
	\medskip
	\noindent\textit{For a root subsystem $\Psi$ of an affine reflection system $\Phi,$ we reserve the notation $Y_\alpha(\Psi)$ (resp. $Y_\alpha'(\Psi)$) or sometime only $Y_\alpha$ (resp. $Y_\alpha'$) for $\alpha\in \mathrm{Gr}(\Psi)$ if the underlying $\Psi$ and the map $p$ are understood.} The next Proposition can be found in \cite{kus2021borel}. 
	
	\begin{prop}\label{propprimeised}Let $\Psi$ be a root subsystem of an affine reflection system $\Phi$. Then the collection $(Y_0':=0, Y_\alpha', \alpha\in \mathrm{Gr}(\Psi))$ is an extension datum. Moreover, the affine reflection system constructed from $(Y_\alpha', \alpha\in \mathrm{Gr}(\Psi)\cup \{0\})$ and $\mathrm{Gr}(\Psi)$ as in \eqref{consaff} is isomorphic to $\Psi\cup\{0\}$.
	\end{prop}
	Let $\Psi=\{\alpha\oplus Y_\alpha(\Psi):\alpha\in \mathrm{Gr}(\Psi)\}$ be a root subsystem of $\Phi.$ Then the dual $\Psi^\vee$ is given by 
	$\Psi^\vee=\{\gamma^\vee:\gamma\in \Psi\}.$ Since $(s_\alpha(\beta))^\vee=s_{\alpha^\vee}(\beta^\vee)$ holds for all $\alpha,\beta\in \Psi,$ it follows that $\Psi$ is a (maximal) root subsystem of $\Phi$ if and only if $\Psi^\vee$ is a (maximal) root subsystem of $\Phi^\vee.$ Moreover, if $p:\mathrm{Gr}(\Psi)\to \Lambda_s$ is a $\bz$-linear funcion such that $p_\alpha\in Y_\alpha(\Psi)$ for all $\alpha\in \mathrm{Gr}(\Psi),$ then  $p^\vee_{\gamma^\vee}:=\frac{2p_\gamma}{(\gamma,\gamma)}$ satisfies $p^\vee_{\gamma^\vee}\in Y_{\gamma^\vee}(\Psi^\vee)$ for all $\gamma^\vee\in \mathrm{Gr}(\Psi)^\vee.$

	Now let $\Phi=\Phi_1\sqcup \Phi_2\sqcup\cdots \sqcup\Phi_r$ be the decomposition of $\Phi$ into irreducible components. Let $\Psi\subseteq \Phi$ be a subset and let 
	\begin{equation}\label{decomppsi}
		\Psi=\Psi_1\sqcup \Psi_2\sqcup \dots \sqcup \Psi_r,\ \ \ \Psi_i=\Psi\cap\Phi_i,\ \ i=1,2,\cdots,r.
	\end{equation}
	We end this subsection with the following easy Lemma.

	\begin{lem}\label{lemassumeirreducible}
		Let $\Psi$ be a subset of $\Phi$ and the decomposition be as in Equation \eqref{decomppsi}. Then
		\begin{enumerate}
			\item $\Psi$ is a root subsystem if and only if $\Psi_i$ is a root subsystem for all $i=1,2,\cdots,r.$
			\item  $\Psi$ is a maximal root subsystem if and only if there exists unique $i$ so that $\Psi_i$ is a maximal root subsystem of $\Phi_i$ and $\Psi_j=\Phi_j$ for all $j\neq i.$
		\end{enumerate}
	\end{lem}

	\begin{rem}
		To characterize the maximal root subsystems of an affine reflection system, it is enough to characterize the maximal root subsystems of an affine reflection system with irreducible gradient. 
	\end{rem}
	
	\subsection{}\label{keysection} 
	From now on we shall only be interested in the irreducible maximal root subsystems of $\Phi.$ \textit{We shall refer this section throughout the article.} Let $\Psi$ be an irreducible root subsystem  of a reduced affine reflection system $\Phi.$ By Proposition \ref{propsametyperoots} and Proposition \ref{propprimeised} we can define $Y_x'=Y_x'(\Psi),x\in\{s,\ell\}.$
	Note that $Y_x',\ x\in\{s,\ell\}$ satisfies 
	\begin{equation}\label{pointedreflectionsubspaces}
		0\in Y_x',\quad Y_x'=-Y_x',\quad Y_x'+2Y_x'\subseteq Y_x',
	\end{equation}
	\begin{equation}\label{primekeyeqn}
		Y_s'+Y_\ell'\subseteq Y_s',\qquad Y_\ell'+mY_s'\subseteq Y_\ell'.
	\end{equation}
	Thus $Y_\ell'$ (resp. $Y_s'$) is a union of cosets of $\langle \Lambda_s \rangle$ modulo $\langle mY_s'\rangle$ (resp. $\langle Y_\ell'\rangle$) containing $0$. Moreover, if $\mathrm{Gr}(\Psi)$ contains a triple $(\alpha,\beta,\alpha+\beta)$ of type $(x,x,x),\ x\in \{s,\ell\}$ then we also have 
	\begin{equation}\label{primeextreaeqns}
		Y_x'+Y_x'\subseteq Y_x',\ \ \ x\in \{s,\ell\}.
	\end{equation}
	In this case $Y_x'$ is an additive subgroup of $Y.$ Conversely, assume that 
	\begin{itemize}
		\item $W_x,\ x\in \{s,\ell\}$ are subspaces of $Y$ satisfying Equations \eqref{pointedreflectionsubspaces} and \eqref{primekeyeqn},
		\item $\mr{\Psi}$ is a root subsystem of $\mr{\Phi},$  $p:\mr{\Psi}\to Y$ is a $\bz$-linear function defined as in Section \ref{sectiondefinitionp} such that $p_\alpha+W_x\subseteq \Lambda_x\ \forall\alpha\in \mr{\Psi}_x,$
	\end{itemize}
	
	\noindent then $\Psi:=\{\alpha\oplus(p_\alpha+W_x):\alpha\in \mr{\Psi}_x,\ x\in\{s,\ell\}\}$ is a root subsystem of $\Phi.$
	
	For a root subsystem $\mr{\Psi}$ of $\mr\Phi$ (including the non-reduced) and a function $p:\mr{\Psi}\to Y$, we say that the data $Y_s',Y_\ell',Y_d'$ define a root subsystem if $\Psi:=\{\alpha\oplus (p_\alpha+Y_\alpha'):\alpha\in \mr{\Psi}\}$ is a root subsystem of $\Phi$ where $Y_\alpha'=Y_x'$ if $\alpha$ is of type $x,\ x\in\{s,\ell,d\}.$ We end this section with a couple of easy Lemmas without proof.
	\begin{lem}\label{redmaxrootsubsystems}
		Let $\mr{\Phi}$ be a reduced irreducible finite root system. We have $\mr{\Phi}_s$ (resp. $\mr{\Phi}_\ell$) is a maximal root subsystem of $\mr{\Phi}$ if and only if $\mr{\Phi}$ is of type $G_2$ or $C_I$ (resp. $\mr{\Phi}$ is of type $G_2$ or $B_I$).
	\end{lem}
	
	\begin{lem}
		Let $J\subseteq I$ be a non-empty subset. Then $\mr{\Psi}_J^B$ defined below is a maximal root subsystem of $B_I;$  $$\mr{\Psi}_J^B:=\{\pm\epsilon_i:i\in I\}\cup \{\pm\epsilon_i\pm\epsilon_j:i,j\in J\text{ or }i,j\notin J\}.$$
		Moreover, any maximal root subsystem $\mr{\Psi}$ of $B_I$ such that $\mr{\Psi}\cap \mr{\Phi}_s\neq \emptyset$ is of the form $\mr{\Psi}_J^B$ for some $\emptyset\neq J\subseteq I.$ Similar statement holds for root system of type $C_I$.
	\end{lem}
	\begin{rem}\label{remfinitebcclosed}
		If $\mr{\Phi}$ is of type $G_2$ or $C_I,$ then the maximal root subsystem $\mr{\Phi}_s$ is not a closed root subsystem of $\mr{\Phi}.$ If $\mr{\Psi}$ is a maximal root subsystem in $B_I$ or $C_I$ such that $\mr{\Psi}\neq \mr{\Phi}_s,\mr{\Phi}_\ell,$ then $\mr{\Psi}$ is a closed root subsystem.
	\end{rem}
	
	\textbf{Mild Assumption:} \textit{For the rest of the paper we will assume that $\Lambda_\ell$ is also a subgroup in the case when $\mathring{\Phi}$ is of type $B_2$ or $BC_2$. Moreover, unless otherwise stated, no irreducible component of affine reflection systems considered, is of type $A_1.$}

	\section{Reduced affine reflection system where \texorpdfstring{$\mr{\Phi}$}{Φ} is not of type \texorpdfstring{$B_I$}{B}}\label{secrednotB}
	In this section we shall classify the maximal root subsystems of an irreducible affine reflection system where $\mr{\Phi}$ is not of type $B_I.$ Note that  $\Lambda_s$ is a subgroup of $Y$ by Section \ref{propexdat}.
	The main result of this section generalizes \cite[Prop 2.6.1(1)]{roy2019maximal}, and \cite{dyer2011reflection,dyer2011root} for maximal root subsystems. We begin with the following Lemma.
	\begin{lem}\label{lemrednotoftypeB}
		Let $\Phi$ be an irreducible reduced affine reflection system such that $\mr{\Phi}$ is not of type $B_I.$ Let $\Psi$ be defined below for some $\bz$-linear function $p:\mr{\Phi}\to Y$ so that $p_\alpha\in Y_\alpha(\Psi)$ for all $\alpha\in\mr{\Phi}.$ Then $\Psi$ is a maximal root subsystem of $\Phi.$
		\begin{enumerate}
			\item\label{lemshortfull} If $\Lambda_\ell\neq m\Lambda_s$ and $\Psi$ is defined by $$\Psi=\{\alpha\oplus \Lambda_s:\alpha\in\mr{\Phi}_s\}\cup\{\alpha\oplus (p_\alpha+S):(p_\alpha+S)\subseteq \Lambda_\ell,\ \alpha\in\mr{\Phi}_\ell\},$$ where $S$ is a maximal subgroup (resp. maximal union of cosets) of $\Lambda_\ell$ (resp. in $\Lambda_s/m\Lambda_s$) if $\mr{\Phi}$ is not of type $C_I$ (resp. is of type $C_I$) containing $m\Lambda_s.$ 
			\item\label{lemdetm} If $\langle \Lambda_\ell\rangle\neq \Lambda_s$ and $\Psi$ is defined by $$\Psi=\{\alpha\oplus (p_\alpha+H):\alpha\in\mr{\Phi}_s\}\cup\{\alpha\oplus \Lambda_\ell:\alpha\in\mr{\Phi}_\ell\},$$ where $H$ is a maximal subgroup of $\Lambda_s$ that contains $\Lambda_\ell.$
			\item\label{lemdetnotm} $H$ is a maximal subgroup of $\Lambda_s$ satisfying $m\Lambda_s+H=\Lambda_s$ and $\Psi$ is defined by $$\Psi:=\{\alpha\oplus (p_\alpha+H):\alpha\in\mr{\Phi}_s\}\cup\{\alpha\oplus (p_\alpha+H)\cap\Lambda_\ell:\alpha\in\mr{\Phi}_\ell\}.$$
		\end{enumerate}
	\end{lem}
	\begin{proof}
		It is easy to check that each $\Psi$ defined above is a root subsystem. We shall only prove $(3)$ since the proofs of $(1)$ and $(2)$ are easy. Assume that $\Psi$ is  given by $(3).$ 
		If $\Psi$ is not maximal in $\Phi,$ then 
		there exists a maximal root subsystem  $\Gamma$ of $\Phi$ satisfying $\Psi\subsetneq \Gamma \subseteq \Phi.$ If $Y_s'(\Gamma)\neq \Lambda_s,$ then $Y_s'(\Gamma)=H$ and Equation \eqref{primekeyeqn} implies that $Y_\beta(\Gamma)\subseteq (p_\beta+H)\cap\Lambda_\ell$ for $\beta\in\mr{\Phi}_\ell.$ Thus $\Gamma\subseteq\Psi,$ a contradiction. Therefore, $Y_s'(\Gamma)= \Lambda_s.$ If $m\neq 2,$ then $\Lambda_\ell$ is a subgroup of $\Lambda_s$ (c.f. Equation \eqref{lambdainclusions}) and by second isomorphism Theorem, we have $\Lambda_\ell/H\cap\Lambda_\ell\cong \Lambda_s/H.$ Therefore $H\cap\Lambda_\ell$ is a maximal subgroup of $\Lambda_\ell.$ If $\Gamma\neq \Phi,$ then 
		$$\Gamma=\{\alpha\oplus\Lambda_\alpha:\alpha\in\mr{\Phi}_s\}\cup \{\alpha\oplus(p_\alpha+H\cap\Lambda_\ell):\alpha\in\mr{\Psi}_\ell\}.$$
		But $\Gamma$ above defines a root subsystem if and only if $m\Lambda_s\subseteq H\cap \Lambda_\ell$ if and only if $m\Lambda_s\subseteq H$ which contradicts the assumption. If $m=2,$ then $\Gamma$ is of the form (\ref{lemshortfull}) for some $S$ and we have $H\cap (\Lambda_\ell-p_\alpha)\subseteq S,\alpha\in\mr{\Phi}_\ell.$ Now $m\Lambda_s\not\subseteq H,$ Lemma \ref{lemmaximalsubgroup} and the fact $0\in\Lambda_\ell-p_\alpha\ \forall\alpha\in\mr{\Phi}_\ell$ imply that all the hypothesis of Lemma \ref{keylemmalattice} are satisfies. Hence we must have $\Lambda_\ell-p_\alpha\subseteq S$ and thus $\Lambda_\ell\subseteq p_\alpha+S, \forall\alpha\in\mr{\Phi}_\ell$ and consequently $\Gamma=\Phi$. In any case, we have that $\Psi$ is a maximal root subsystem of $\Phi.$
	\end{proof}
	\begin{rem}\label{spremark}
		\begin{enumerate}[leftmargin=*]
			\item If $\Lambda_\ell$ is a subgroup of $\Lambda_s,$ then we have $(p_\alpha+H)\cap \Lambda_\ell=p_\alpha+(H\cap\Lambda_\ell),\ \forall\alpha\in\mr{\Phi}_\ell.$
			\item Since $\Lambda_\ell$ is a union of cosets of $\Lambda_s$ modulo $m\Lambda_s,$ if $m=2,$ we have $m\Lambda_s\subseteq \Lambda_\ell+a$ for all $a\in\Lambda_\ell.$ Therefore if $m\Lambda_s\not\subseteq H,$ then $\Lambda_\ell+a\not\subseteq H$ for any $a\in \Lambda_\ell.$ In particular $(p_\alpha+H)\cap \Lambda_\ell\neq \Lambda_\ell$ for all $\alpha\in\mathrm{\Phi}_\ell$ and consequently $Y_\alpha(\Psi)\neq \Lambda_\ell$ for all $\alpha\in\mr{\Phi}_\ell$ if $\Psi$ is defined as in Lemma \ref{lemrednotoftypeB}(\ref{lemdetnotm}).
		\end{enumerate}
	\end{rem}
	\noindent The next Proposition proves the converse of the above Lemma when $\mathrm{Gr}(\Psi)$ is full.
	\begin{prop}\label{proprednotoftypeB}
		Let $\Phi$ be an irreducible reduced affine reflection system such that $\mr{\Phi}$ is not of type $B_I$. Let $\Psi$ be a maximal root subsystem of $\Phi.$ Then one of the following holds. 
		\begin{enumerate}
			\item\label{grproper} $\mathrm{Gr}(\Psi)$ is a maximal root subsystem of $\mr{\Phi}$ and $Y_\alpha(\Psi)=\Lambda_\alpha$ for all $\alpha\in \mathrm{Gr}(\Psi).$
			\item \label{subrootnotb} $\mathrm{Gr}(\Psi)=\mr{\Phi}$ and there exist a $\bz$-linear function $p:\mathrm{Gr}(\Psi)\to Y$ satisfying $p_\alpha\in Y_\alpha(\Psi)\ \forall \alpha\in \mr{\Phi}$ and $\Psi$ is given by Lemma \ref{lemrednotoftypeB}.
		\end{enumerate}
	\end{prop}
	
	\begin{proof}
		Let $\Psi$ be a maximal root subsystem of $\Phi.$ If $\mathrm{Gr}(\Psi)$ is proper in $\mr{\Phi},$ then clearly $\mathrm{Gr}(\Psi)$ is a maximal root subsystem of $\mr{\Phi}$ and $Y_\alpha(\Psi)=\Lambda_\alpha\ \forall\alpha\in \mathrm{Gr}(\Psi).$ So assume that $\mathrm{Gr}(\Psi)=\mr{\Phi}.$ By Section \ref{sectiondefinitionp} there exists a $\bz$-linear function $p:\mr{\Phi}\to Y$ such that $p_\alpha\in Y_\alpha(\Psi)$ for all $\alpha\in \mathrm{Gr}(\Psi).$ Since $\Psi$ is irreducible, we define $Y_s'$ and $Y_\ell'$ as in Section \ref{keysection}.
		\medskip
		
		\noindent First we assume that $Y_s'=\Lambda_s.$ By Section \ref{keysection} we have that $Y_\ell'$ is a union of cosets in $\Lambda_s/m\Lambda_s$ containing $m\Lambda_s$ as one of its cosets. Note that $\Lambda_\ell\neq m\Lambda_s$ must hold. 
		However, if $m=2$ and $A$ is any union of cosets in $\Lambda_s/2\Lambda_s$ containing $2\Lambda_s,$ then we must have $$A=-A,\ \ 0\in A,\ \ A+2A\subseteq A,\ \ A+2\Lambda_s\subseteq A,\ \ \Lambda_s+2A\subseteq \Lambda_s.$$ Hence it follows that $Y_\ell'$ is a maximal union of cosets in $\Lambda_s/m\Lambda_s$ satisfying $p_\alpha+Y_\ell'\subseteq \Lambda_\ell,\ \alpha\in\mr{\Phi}_\ell.$ If $m\neq 2,$ then by Equation \eqref{primeextreaeqns} we have that $Y_\ell'$ is a maximal subgroup of $\Lambda_\ell$ containing $m\Lambda_s.$ In particular, $\Psi$ is of the form Lemma \ref{lemrednotoftypeB}(\ref{lemshortfull}).
		\medskip
		
		\noindent Now assume that $Y_s'\neq \Lambda_s.$ Recall that $Y_s'$ is a subgroup of $\Lambda_s.$ Let $H$ be a subgroup of $\Lambda_s$ containing $Y_s'.$  Define $\Psi'$ by $$\Psi':=\{\alpha\oplus(p_\alpha+H):\alpha\in \mr{\Phi}_s\}\cup \{\alpha\oplus (p_\alpha+H)\cap \Lambda_\ell:\alpha\in \mr{\Phi}_\ell\}.$$
		We shall show that $\Psi'$ is a root subsystem of $\Phi.$ Assume that for the moment. Since $Y_\ell'(\Psi)\subseteq Y_s'(\Psi)\subseteq H,$ we have that $p_\alpha+Y_\ell'(\Psi)\subseteq (p_\alpha+H)\cap\Lambda_\ell.$ Hence $\Psi'$ satisfies $\Psi\subseteq \Psi'\subsetneq \Phi.$ Since $\Psi$ is a maximal root subsystem of $\Phi,$ it follows that $\Psi=\Psi'$ and $H$ is a maximal subgroup of $\Lambda_s.$ 
		
		Now we prove that $\Psi'$ is a root subsystem of $\Phi.$ Let $\alpha_1=\alpha\oplus (p_\alpha+h)\in\Psi'$ and $\beta_1=\beta\oplus (p_\beta+h_1)\in\Psi'.$ Then we have $$s_{\alpha_1}(\beta_1)=s_\alpha(\beta)\oplus[(p_\beta+h_1)-(\beta,\alpha^\vee)(p_\alpha+h)]=s_\alpha(\beta)\oplus [p_{s_\alpha(\beta)}+(h_1-(\beta,\alpha^\vee)h)].$$
		If $\beta\in\mr{\Phi}_s,$ then it is clear that $s_{\alpha_1}(\beta_1)\in \Psi'.$ Suppose that $\beta\in\mr{\Phi}_\ell.$ If $\alpha$ is short, then $(p_\beta+h_1)-(\beta,\alpha^\vee)(p_\alpha+h)\in p_{s_\alpha(\beta)}+H.$ Moreover, we have $$(p_\beta+h_1)-(\beta,\alpha^\vee)(p_\alpha+h)\in \Lambda_\ell+m\Lambda_s\subseteq \Lambda_\ell.$$ Therefore $s_{\alpha_1}(\beta_1)\in \Psi'.$ Now assume that $\alpha$ is also a long root. If $\Lambda_\ell$ is a subgroup of $Y,$ then the same argument along with Remark \ref{spremark} implies that $s_{\alpha_1}(\beta_1)\in \Psi'.$ If $\Lambda_\ell$ is not a subgroup, then we have $m=2$ and $\mr{\Phi}$ is of type $C_I.$ Hence $s_{\alpha_1}(\beta_1)\neq \beta_1$ if and only if $\beta=\pm\alpha.$ In that case, we need to show that $C_\alpha:=(p_\alpha+H)\cap\Lambda_\ell$ satisfies $$C_{\alpha}- 2C_{\alpha}\subseteq C_{-\alpha},\  (\text{or equivalently})\ \  C_{-\alpha}+ 2C_{\alpha}\subseteq C_{\alpha},$$ which holds since $\Lambda_\ell$ satisfies relations in Section \ref{lambdainclusions}. This proves that $\Psi'$ is a root subsystem of $\Phi.$ Summarizing, there exists a maximal subgroup $H$ of $\Lambda_s$ such that $\Psi$ is of the form 
		$$\Psi=\{\alpha\oplus(p_\alpha+H):\alpha\in \mr{\Phi}_s\}\cup \{\alpha\oplus (p_\alpha+H)\cap \Lambda_\ell:\alpha\in \mr{\Phi}_\ell\}.$$
		
		\medskip
		
		\textit{Case 1:}  We shall first consider the case when $m\Lambda_s\subseteq H.$ It is easy to check that $\Psi'$ defined below is a root subsystem of $\Phi$ which contains $\Psi:$
		$$\Psi':=\{\alpha\oplus(p_\alpha+\Lambda_s):\alpha\in\mr{\Phi}_s\}\cup \{\alpha\oplus(p_\alpha+H)\cap\Lambda_\ell:\alpha\in\mr{\Phi}_\ell\},$$
		Hence $\Psi'=\Phi$ and we have $(p_\alpha+H)\cap\Lambda_\ell=\Lambda_\ell$ for all $\alpha\in\mr{\Phi}_\ell$ and so $\Psi$ is of the  form $$\Psi=\{\alpha\oplus (p_\alpha+H):\alpha\in \mr{\Phi}_s\}\cup \{\alpha\oplus \Lambda_\ell:\alpha\in\mr{\Phi}_\ell\}.$$
		Therefore we have $H+\Lambda_\ell=H$ by Equation \eqref{primekeyeqn} and thus $\Lambda_\ell\subseteq H.$ Note that $\langle\Lambda_\ell\rangle \neq \Lambda_s$ must hold and so $\Psi$ is of the form Lemma \ref{lemrednotoftypeB}(\ref{lemdetm}).
		
		\medskip
		
		\textit{Case 2:} Now assume that  $m\Lambda_s\not\subseteq H.$ Indeed $H$ satisfies $\langle\Lambda_\ell\rangle+H=\Lambda_s$ otherwise $\Psi''$ defined below is a proper root subsystem of $\Phi$ which strictly contains $\Psi:$
		$$\Psi'':=\{\alpha\oplus(p_\alpha+\langle\Lambda_\ell\rangle+H):\alpha\in\mr{\Phi}_s\}\cup \{\alpha\oplus\Lambda_\ell:\alpha\in\mr{\Phi}_\ell\}.$$ In particular, $\Psi$ is of the form Lemma \ref{lemrednotoftypeB}(\ref{lemdetnotm}). This completes the proof.
	\end{proof}

	\noindent We are now ready to prove the main result of this section.
	
	\begin{thm}\label{thmrednotoftypeB}
		Let $\Phi$ be an irreducible reduced affine reflection system such that $\mr{\Phi}$ is not of type $B_I.$ A subset $\Psi$ of $\Phi$ is a maximal root subsystem of $\Phi$ if and only if one of the following holds.
		\begin{enumerate}
			\item $\Psi$ is defined as in Lemma \ref{lemrednotoftypeB}.
			\item $\Psi$ is defined as in Proposition \ref{proprednotoftypeB}(\ref{grproper}) and one of the following holds:
			\begin{enumerate}
				\item $\mathrm{Gr}(\Psi)\neq \mr{\Phi}_s,\mr{\Phi}_\ell.$ 
				\item $\mathrm{Gr}(\Psi)=\mr{\Phi}_s$ and $\Lambda_\ell= m\Lambda_s.$
				\item\label{longmaximal} $\mathrm{Gr}(\Psi)=\mr{\Phi}_\ell$ and $\langle\Lambda_\ell\rangle= \Lambda_s.$
			\end{enumerate}

		\end{enumerate}
	\end{thm}
	\begin{proof}
		$(1)$ follows from Lemma \ref{lemrednotoftypeB} and Proposition \ref{proprednotoftypeB}. The forward direction of $(2)$ follows from Lemma \ref{lemrednotoftypeB}. To prove the converse, assume that $\Psi$ be defined by Proposition \ref{proprednotoftypeB}(\ref{grproper}). If $\mathrm{Gr}(\Psi)\neq \mr{\Phi}_s,\mr{\Phi}_\ell,$ then $\mathrm{Gr}(\Psi)$ contains both short and long roots. Hence any root subsystem $\Gamma$ containing $\Psi$ must satisfy $\Gamma=\Phi$ by  Proposition \ref{proprednotoftypeB}(\ref{subrootnotb}), Lemma \ref{lemrednotoftypeB} and Remark \ref{spremark}(2) and thus proving $(a).$ Similar arguments apply to $(b)$ and $(c).$
	\end{proof}
	\begin{cor}
		If $\mr{\Phi}_s$ is a maximal root subsystem of $\mr{\Phi}$, then the lift $\widehat{\mr{\Phi}_s}$ is a maximal root subsystem if  and only if $\Lambda_\ell=m\Lambda_s.$ Similarly, If $\mr{\Phi}_\ell$ is a maximal root subsystem of $\mr{\Phi}$, then the lift $\widehat{\mr{\Phi}_\ell}$ is a maximal root subsystem if  and only if $\langle\Lambda_\ell\rangle=\Lambda_s.$
	\end{cor}

	\section{Reduced affine reflection systems where \texorpdfstring{$\mr{\Phi}$}{Φ} is of type \texorpdfstring{$B_I$}{B}}\label{secredB}
	In this section we shall assume the case where $\mr{\Phi}$ is of type $B_I.$ Recall that $\Lambda_\ell$ is a subgroup of $Y$ but $\Lambda_s$ may not be a subgroup. We begin with the following Lemma.
	\begin{lem}\label{lemredoftypeB}
		Let $\Phi$ be an irreducible affine reflection system such that $\mr{\Phi}$ is of type $B_I.$ Let $\Psi$ be defined below for some $\bz$-linear function $p:\mr{\Phi}\to Y$ so that $p_\alpha\in Y_\alpha(\Psi)$ for all $\alpha\in\mr{\Phi}.$ Then $\Psi$ is a maximal root subsystem of $\Phi.$
		\begin{enumerate}
			\item\label{lemlongfullb} $\Lambda_s\neq \Lambda_\ell$ and $\Psi$ is given by
			$$\Psi=\{\alpha\oplus(p_\alpha+ S):p_\alpha+ S\subseteq \Lambda_s,\ \ \alpha\in \mr{\Phi}_s\}\cup \{\alpha\oplus \Lambda_\ell:\alpha\in\mr{\Phi}_\ell\},$$
			where $S$ is a  maximal union of cosets in $\langle \Lambda_s\rangle /\Lambda_\ell$ and  $S\subseteq \Lambda_s$ if $\Lambda_s$ is a subgroup of $Y$.
			
			\item\label{lemdetmb}  $\Lambda_\ell\neq \langle 2\Lambda_s\rangle$ and $\Psi$ is given by $$\Psi=\{\alpha\oplus \Lambda_s:\alpha\in \mr{\Phi}_s\}\cup \{\alpha\oplus (p_\alpha+H):\alpha\in\mr{\Phi}_\ell\},$$ where $H$ is a maximal subgroup of $\Lambda_\ell$ such that $2\Lambda_s\subseteq H.$ 
			\item\label{lemdetnotmb} $\Psi$ is given by 
			$$\Psi=\{\alpha\oplus (p_\alpha+S):p_\alpha+S\subseteq \Lambda_s,\ \alpha\in\mr{\Phi}_s\}\cup\{\alpha\oplus (p_\alpha+H):\alpha\in\mr{\Phi}_\ell\}$$
			where $H$ is a maximal subgroup of $\Lambda_\ell$ satisfying $H+\langle 2\Lambda_s\rangle=\Lambda_\ell$ and $S$ is maximal among the subsets of $\langle\Lambda_s \rangle$ 
			\begin{itemize}
				\item $S$ is a union of cosets in $\langle \Lambda_s \rangle/H$ and $H\subseteq S,$
				\item $2S\subseteq H.$ 
			\end{itemize}
			Moreover, $S$ is a subgroup of $\Lambda_s$ if $\Lambda_s$ is a group.
		\end{enumerate}
	\end{lem}
	\begin{example}\label{funnyex}
		If $\Psi$ is given by Lemma \ref{lemredoftypeB}(\ref{lemdetnotmb}), then $S$ need not be a subgroup of $\langle\Lambda_s\rangle.$ Let $\Lambda_\ell=2\bz\times 2\bz,\Lambda_s=\{(0,0),(1,0),(0,1)\}+\Lambda_\ell.$ Then $H$ and $S$ given below (with $p_\alpha=0$ for all $\alpha\in\mr{\Phi}$) satisfy all the required properties:
		$$H=2\bz\times 4\bz,\ \ \ S=\Lambda_\ell\cup (\{(1,0),(1,2)\}+H).$$
		We shall classify $S$ for rank $2$ in Section \ref{secnul2}. Also note that any $S$ given by Lemma \ref{lemredoftypeB}(\ref{lemdetnotmb}) satisfies $S=-S$ and is a proper subset of $\Lambda_s$ since $2\Lambda_s\not\subseteq H$.
	\end{example}

	\begin{proof}
		Only the proof of $(3)$ is non-trivial and the proof of the first part of $(3)$ is similar to Lemma \ref{lemrednotoftypeB}(\ref{lemdetnotm}) using the fact that any maximal root subsystem that contains (\ref{lemdetnotmb}) is of the form (\ref{lemlongfullb}). Now assume that $\Lambda_s$ is a group. Since $S$ is maximal, it follows that $S$ is precisely the union of those cosets $a+H$ such that $2a\in H.$ Moreover, for any two cosets $a+H,b+H\subseteq S,$ their sum satisfies $2(a+b)+H\subseteq H.$ Therefore, $S$ is closed under addition and thus $S$ is a subgroup of $\Lambda_s$.
	\end{proof}
	\begin{rem}
		Note that if $\Psi$ is given by Lemma \ref{lemredoftypeB}(\ref{lemlongfullb}), (\ref{lemdetmb}) and (\ref{lemdetnotmb}), then the dual $\Psi^\vee$ is given by Lemma \ref{lemrednotoftypeB}(\ref{lemshortfull}), (\ref{lemdetm}) and (\ref{lemdetnotm}) respectively.
	\end{rem}
	
	The next Proposition provides the necessary conditions a maximal root subsystem must satisfy. It in turn proves the converse of Lemma \ref{lemredoftypeB} when $\mathrm{Gr}(\Psi)=\mr{\Phi}.$
	\begin{prop}\label{propredoftypeB}
		Let $\Phi$ be an irreducible affine reflection system such that $\mr{\Phi}$ is of type $B_I.$ Let $\Psi$ be a maximal root subsystem of $\Phi.$ Then exactly one of the following holds.
		\begin{enumerate}
			\item\label{grproperb} $\mathrm{Gr}(\Psi)$ is a maximal root subsystem of $\mr{\Phi}$ and $Y_\alpha(\Psi)=\Lambda_\alpha\ \forall \alpha\in \mathrm{Gr}(\Psi).$
			\item\label{subrootb} $\mathrm{Gr}(\Psi)=\mr{\Phi},$ there exist a $\bz$-linear function $p:\mathrm{Gr}(\Psi)\to Y$ satisfying $p_\alpha\in Y_\alpha(\Psi),\ \forall \alpha\in \mathrm{Gr}(\Psi)$ and $\Psi$ is given by Lemma \ref{lemredoftypeB}.
		\end{enumerate}
	\end{prop}
	\begin{proof}
		Let $\Psi$ be a maximal root subsystem of $\Phi.$ If $\mathrm{Gr}(\Psi)$ is proper in $\mr{\Phi},$ then obviously $\mathrm{Gr}(\Psi)$ is a maximal root subsystem of $\mr{\Phi}$ and $Y_\alpha(\Psi)=\Lambda_\alpha\ \forall\alpha\in \mathrm{Gr}(\Psi).$ So assume that $\mathrm{Gr}(\Psi)=\mr{\Phi}.$ We define a $\bz$-linear function $p:\mr{\Phi}\to \Lambda_s$ and $Y_s',Y_\ell'$ as in Section \ref{keysection}. Note that $Y_\ell'$ is always a subgroup of $\Lambda_\ell.$\medskip
		
		First assume that $Y_\ell'=\Lambda_\ell.$ We have that $Y_s'$ is a union of cosets in $\langle\Lambda_s\rangle/\Lambda_\ell$ containing $\Lambda_\ell$ and hence $\Lambda_\ell\neq\Lambda_s.$  Now $\Psi'$ defined below is a root subsystem which satisfies $\Psi\subseteq \Psi':$
		$$\Psi':=\{\alpha\oplus S:\alpha\in \mr{\Phi}_s\}\cup \{\alpha\oplus\Lambda_\ell:\alpha\in \mr{\Phi}_\ell\},$$
		where $S$ is a maximal union of cosets in $\Lambda_s/\Lambda_\ell$ so that $\Lambda_\ell\subseteq S.$ Therefore $\Psi=\Psi'.$ In addition this case occurs only if $\Lambda_\ell\neq \Lambda_s$ and thus $\Psi$ is of the form Lemma \ref{lemredoftypeB}(\ref{lemlongfullb}).\medskip
		
		\noindent Now assume that $Y_\ell'=H,$ a proper subgroup of $\Lambda_\ell.$ Define a root subsystem $\Psi'$ containing $\Psi$ by 
		$$\Psi':=\{\alpha\oplus \Lambda_s:\alpha\in \mr{\Phi}_s\}\cup \{\alpha\oplus (p_\alpha+H+\langle 2\Lambda_s\rangle):\alpha\in\mr{\Phi_\ell}\}.$$

		\textit{Case 1:} If $\Psi'=\Psi,$ then we have $2\Lambda_s\subseteq H$ and $H$ is a maximal subgroup of $\Lambda_\ell.$ Hence we must have $\Lambda_\ell\neq \langle 2\Lambda_s\rangle.$ In particular, $\Psi$ is of the form Lemma \ref{lemredoftypeB}(\ref{lemdetmb}).
		
		\textit{Case 2:} If $\Psi'=\Phi,$ then the subgroup $H$ satisfies $H+\langle 2\Lambda_s\rangle=\Lambda_\ell.$ We claim that $H$ is a maximal subgroup of $\Lambda_\ell$ satisfying Lemma \ref{lemredoftypeB}(\ref{lemdetnotmb}). If possible assume that there exists a proper subgroup $H'$ of $\Lambda_\ell$ such that $H\subsetneq H'.$ Since $Y_s'(\Psi)$ is a union of cosets in $\langle \Lambda_s\rangle/H$ containing $H,$ (c.f. Equation \eqref{primekeyeqn}),we can write $Y_s'(\Psi)=\cup_{x\in A} (x+H)$ for some $A\subseteq\langle\Lambda_s\rangle$ such that $0\in A.$  Moreover, $2Y_s'\subseteq H$ implies that $2x\in H$ for all $x\in A.$ Now we define $Y_s''=\cup_{x\in A} (x+H').$ Note that we must have $p_\alpha+Y_s''\subseteq \Lambda_s,\ \forall \alpha\in\mr{\Phi}_s$ since $p_\alpha+Y_s'(\Psi)\subseteq \Lambda_s,\ \forall \alpha\in\mr{\Phi}_s$ and $\Lambda_s+\Lambda_\ell\subseteq \Lambda_s.$ Now $2x\in H\subseteq H'$ for all $x\in A$ implies
		\begin{equation}\label{eq0}
			2Y_s''=\cup (2x+2H')\subseteq \cup (2x+H')=H'.
		\end{equation}
		
		\noindent Therefore $Y_s''+2Y_s''\subseteq Y_s''.$ Now define $\Psi''$ by 
		$$\Psi''=\{\alpha\oplus (p_\alpha+Y_s''):\alpha\in\mr{\Phi}_s\}\cup\{\alpha\oplus (p_\alpha+H'):\alpha\in \mr{\Phi}_\ell\}.$$
		Clearly $\Psi\subsetneq \Psi''\subsetneq \Phi$ and the following relations show that (the other relations are easy to verify) $\Psi''$ is a root subsystem of $\Phi:$
		$$Y_s''+2Y_s''\subseteq Y_s'',\ \ \ Y_s''+H'\subseteq Y_s'', \ \ \ H'+2Y_s''\subseteq H'+H'\subseteq H'\ \ (\text{by Equation}\eqref{eq0})$$ 
		which contradicts the fact that $\Psi$ is maximal and the claim follows. Now it is clear that $\Psi$ is of the form Lemma \ref{lemredoftypeB}(\ref{lemdetnotmb}). This completes the proof.
	\end{proof}
	
	\noindent We are now ready to prove the main result of this section.
	\begin{thm}\label{thmredoftypeB}
		Let $\Phi$ be an irreducible affine reflection system such that $\mr{\Phi}$ is of type $B_I.$ A subset $\Psi$ of $\Phi$ is a maximal root subsystem of $\Phi$ if and only if one of the following holds.
		\begin{enumerate}
			\item $\Psi$ is defined by Proposition \ref{propredoftypeB}(\ref{subrootb}).
			\item $\mathrm{Gr}(\Psi)$ is a maximal root subsystem of $\mr{\Phi}$ and one of the following holds:
			\begin{enumerate}
				\item\label{opa} $\mathrm{Gr}(\Psi)= \mr{\Phi}_\ell$ and  $\Lambda_\ell=\Lambda_s.$
				\item\label{opb} $\mathrm{Gr}(\Psi)\neq \mr{\Phi}_\ell$
			\end{enumerate}
		\end{enumerate}
	\end{thm}
	\begin{proof}
		$(1)$ follows from the proof of Lemma \ref{lemredoftypeB} and Proposition \ref{propredoftypeB} and $2(a)$ follows from Lemma \ref{lemredoftypeB}(\ref{lemlongfullb}). If $\mathrm{Gr}(\Psi)\neq \mr{\Phi}_\ell,$ then $\mathrm{Gr}(\Psi)$ contains both long and short roots by Lemma \ref{redmaxrootsubsystems}. The proof of $2(b)$ now follows from Proposition \ref{propredoftypeB}.
	\end{proof}
	
	\subsection{Summary for the reduced case}
	Table \ref{Tablereduced} summarizes the closeness property of root subsystems that appeared in Section \ref{secrednotB} and Section \ref{secredB}. Recall the definition of $\Psi^\vee$ in Section \ref{sectiondefinitionp}. Note that if $\Psi$ or $\Psi^\vee$ is (real) closed in $\Phi$ or $\Phi^\vee$, then they are maximal closed root subsystem of $\Phi$ or $\Phi^\vee$ respectively and thus they appear in \cite[Theorem 3]{kus2021borel}. We investigate the properties of $\Psi$ case by case. 
	\begin{enumerate}[leftmargin=*]
		\item If $\Psi$ is of the form Lemma \ref{lemrednotoftypeB}(\ref{lemshortfull}): $\Psi$ is not closed since a maximal closed root subsystem $\Psi$ satisfies $Y_\alpha(\Psi)\neq \Lambda_s$ for any $\alpha\in\mr{\Phi}_s$.
		
		\item If $\Psi$ is of the form Lemma \ref{lemrednotoftypeB}(\ref{lemdetm}): Since $\Lambda_\ell\subseteq H$ we have that $(a+H)\cap \Lambda_\ell=\Lambda_\ell$ for all $a\in \Lambda_\ell.$ Hence $\Psi$ is closed.

		\item If $\Psi$ is of the form Lemma \ref{lemrednotoftypeB}(\ref{lemdetnotm}): $\Psi$ is closed follows directly from \cite[Theorem 3.3.(i)]{kus2021borel}.
		
		\item If $\Psi$ is of the form Theorem \ref{thmrednotoftypeB}(2a): If $\Phi$ is of type $C_I,$ then $\Psi$ is closed by Remark \ref{remfinitebcclosed}. Otherwise we have that $\mr{\Phi}\cong \mr{\Phi}^\vee$ by  Section \ref{sectionduality} and hence $\mathrm{Gr}(\Psi)$ is closed in $\mr{\Phi}$ by \cite[Corollary 2]{dyer2011reflection} and thus $\Psi$ is closed.
		\item If $\Psi$ is of the form Theorem \ref{thmrednotoftypeB}(2b): Choose $\alpha,\beta\in \mr{\Phi}_s$ such that $\alpha+\beta\in\mr{\Phi}_\ell$ and $a\in \Lambda_\ell\subseteq \Lambda_s.$ Then $\alpha+\beta\oplus a\in \Phi\backslash \Psi.$ Hence $\Psi$ is not closed.
		\item If $\Psi$ is of the form Theorem \ref{thmrednotoftypeB}(2c): Since the sum of two long roots is a long root, it follows that $\Psi$ is a closed root subsystem of $\Phi.$
		\item $\Psi$ is of the form Theorem \ref{lemredoftypeB}(\ref{lemlongfullb}): Clearly $\Psi$ is closed.
		\item If $\Psi$ is of the form Theorem \ref{lemredoftypeB}(\ref{lemdetmb}): $\Psi$ is not closed. If $\Psi$ is closed, then by \cite[Theorem 3.3.(ii)]{kus2021borel} there exists a maximal subgroup $H$ of $\Lambda_\ell$ such that $Y_\alpha(\Psi)=\Lambda_s=p_\alpha+Z$ where $Z$ is a union of cosets in $\langle \Lambda_s\rangle/H$ maximal with respect to the properties $H\subseteq Z,\ Z=-Z,\ (Z+Z)\cap\Lambda_\ell\subseteq H.$ Since $\Lambda_s$ is a union of cosets in $\langle\Lambda_s\rangle/\Lambda_\ell$ and $Z$ is a translate of $\Lambda_s,$ we can also write $Z$ as a union of cosets in $\langle\Lambda_s\rangle/\Lambda_\ell$. Since $\langle 2\Lambda_s\rangle\subseteq \Lambda_\ell,$ we have $\Lambda_\ell\subseteq (Z+Z)$ and therefore $\Lambda_\ell\subseteq (Z+Z)\cap\Lambda_\ell\subseteq H\subseteq \Lambda_\ell,$ a contradiction. 
		\item If $\Psi$ is of the form Lemma \ref{lemredoftypeB}(\ref{lemdetnotmb}): Consider Example \ref{funnyex}. Note that $(2,2)\in (S+S)\cap \Lambda_\ell$ but $(2,2)\notin H.$ In particular, $\Psi$ is not closed by \cite[Theorem 3(3)]{kus2021borel} .
	\end{enumerate}

	\begin{center}
		\begin{table}[ht]
			{\renewcommand{\arraystretch}{1.5}
				\renewcommand{\tabcolsep}{8pt}
				\begin{NiceTabular}{|c|c|c|c|c|}[hvlines]
					The form of $\Psi$ & $\Phi=C_I$ & The form of $\Psi^\vee$ &$\Psi$ is closed & $\Psi^\vee$ is closed \\
					\Block{2-1}{Lemma \ref{lemrednotoftypeB}(\ref{lemshortfull})} & \checkmark &Lemma \ref{lemredoftypeB}(\ref{lemlongfullb})  & \Block{2-1}{\xmark}  & \Block{2-1}{\checkmark}  \\
					
					& \xmark  & Lemma \ref{lemrednotoftypeB}(\ref{lemdetm}) &  & \\
					
					\Block{2-1}{Lemma \ref{lemrednotoftypeB}(\ref{lemdetm})} & \checkmark & Lemma \ref{lemredoftypeB}(\ref{lemdetmb})  &\Block{2-1}{\checkmark}  & \Block{2-1}{\xmark}  \\
					
					& \xmark & Lemma \ref{lemrednotoftypeB}(\ref{lemshortfull}) & &  \\
					
					\Block{2-1}{Lemma \ref{lemrednotoftypeB}(\ref{lemdetnotm})} &\checkmark & Lemma \ref{lemredoftypeB}(\ref{lemdetnotmb})  &\Block{2-1}{\checkmark}  & \xmark  \\
					
					&\xmark & Lemma \ref{lemrednotoftypeB}(\ref{lemdetnotm})  & &\checkmark \\
					
					\Block{2-1}{Theorem \ref{thmrednotoftypeB}(2a)} &\checkmark  & Theorem \ref{thmredoftypeB}(2b) &\Block{2-1}{\checkmark}  & \Block{2-1}{\checkmark}\\
					
					&\xmark  &Theorem \ref{thmrednotoftypeB}(2a)   &   &  \\
					
					\Block{2-1}{Theorem \ref{thmrednotoftypeB}(2b)}&\checkmark  & Theorem \ref{thmredoftypeB}(2a) &\Block{2-1}{\xmark}  &\Block{2-1}{\checkmark} \\
					
					&\xmark  &Theorem \ref{thmrednotoftypeB}(2c)   &   &  \\
					Theorem \ref{thmrednotoftypeB}(2c)&\xmark &Theorem \ref{thmrednotoftypeB}(2b) & \checkmark  &  \xmark
			\end{NiceTabular}}
			\caption{Maximal root subsystems of a reduced affine reflection system}
			\label{Tablereduced}
		\end{table}
	\end{center}
	
	We now have the analogous Theorem to \cite[Corollary 2]{dyer2011reflection}.
	\begin{thm}\label{thmclsdordualclsd}
		Let $\Psi$ be a maximal root subsystem of a reduced affine reflection system $\Phi$ (including type $A_1$). Then either $\Psi$ is closed in $\Phi$ or $\Psi^\vee$ is closed in $\Phi^\vee.$
	\end{thm}
	\begin{proof}
		Using Lemma \ref{lemassumeirreducible}, we can assume that $\Phi$ is irreducible. If $\mr{\Phi}$ is of type $A_1,$ the clearly $\Psi$ is closed in $\Phi.$ If $\mr{\Phi}$ is not of type $A_1$, the result follows from Table \ref{Tablereduced}.
	\end{proof}
	
	\section{Non-reduced affine reflection systems}\label{secnonred} In this section we shall assume that $\Phi$ is an affine reflection system such that $\mr{\Phi}$ is non-reduced i.e. $\mr{\Phi}$ is of type $BC_I.$ In this case, we have $$\mr{\Phi}_s=\{\pm\epsilon_i:i\in I\}, \ \ \mr{\Phi}_\ell=\{\pm\epsilon_i\pm\epsilon_j:i\neq j\in I\}, \ \ \mr{\Phi}_d=\{\pm 2\epsilon_i:i\in I\}, \ \ $$ By Section \ref{lambdainclusions} the following relations hold. 
	\begin{align*}
		\Lambda_\ell + 2 \Lambda_s \subseteq \Lambda_\ell,  &\hspace{30pt}
		\Lambda_s + \Lambda_\ell \subseteq \Lambda_s, \\
		\Lambda_d + 2 \Lambda_\ell \subseteq \Lambda_d,  &\hspace{30pt}
		\Lambda_\ell +\Lambda_d \subseteq \Lambda_\ell, \\
		\Lambda_d + 4 \Lambda_s \subseteq \Lambda_d,  &\hspace{30pt}
		\Lambda_s + \Lambda_d \subseteq \Lambda_s.
	\end{align*}
	
	Following \cite{kus2021borel,roy2019maximal} we define the following root subsystems of $\mr{\Phi}$. Let $J\subseteq I.$
	\begin{align*}
		A_J:=&\{\pm 2\epsilon_i:i\in I\}\cup \{\pm\epsilon_j:j\in J\}\cup \{\pm\epsilon_k\pm\epsilon_\ell:k,\ell\in J\text{ or }k,\ell\notin J\},\\
		B^I:=&\{\pm \epsilon_i:i\in I\}\cup\{\pm\epsilon_i\pm\epsilon_j:i,j\in I\}.
	\end{align*}
	Note that $A_{\emptyset}=C_I.$ We also define $$\Phi_B:=\{\pm\epsilon_i\oplus\Lambda_s:1\le i\le n\}\cup \{\pm\epsilon_i\pm\epsilon_j\oplus\Lambda_\ell:1\le i\neq j\le n\} =\widehat{B^I},$$
	$$\Phi_C:=\{\pm2\epsilon_i\oplus\Lambda_d:1\le i\le n\}\cup \{\pm\epsilon_i\pm\epsilon_j\oplus\Lambda_\ell:1\le i\neq j\le n\}=\widehat{A_\emptyset},$$
	$$\mr{\Phi}_{xy}:=\mr{\Phi}_x\cup \mr{\Phi}_y,\ \ \ x,y\in \{s,\ell,d\}.$$
	
	\noindent The next Lemma characterizes the maximal root subsystems of $\mr{\Phi}.$ 
	\begin{lem}\label{nonredgrmaximal}
		Maximal root subsystems of $BC_n$ are precisely $B^I$ and $A_J$ for some proper subset $J$ of $I.$
	\end{lem}
	\begin{proof}
		It is easy to check that
		$A_J$ and $B_I$ defined above are maximal root subsystems of $\mr{\Phi}.$ 
		Conversely assume that $\mr{\Psi}$ is a maximal root subsystem of $\mr{\Phi}.$ Let $I_{\mr{\Psi}}:=\{i\in I:\epsilon_i\in \mr{\Psi}\}.$ If $I_{\mr{\Psi}}\subsetneq I,$ then we have that $\mr{\Psi}\subseteq A_{I_{\mr{\Psi}}}$ since $s_{\epsilon_i-\epsilon_j}(\epsilon_i)=\epsilon_j.$ It follows that $\mr{\Psi}=A_{I_{\mr{\Psi}}}.$ If $I_{\mr{\Psi}}=I,$ similar argument proves that $\mr{\Psi}\subseteq B^I$ and hence they are equal.
	\end{proof}
	\begin{rem}\label{AJhatmax}
		The root subsystem $\widehat{A_J}$ is a maximal root subsystem of $\Phi$ where $\emptyset\neq J\subsetneq I.$
	\end{rem}
	
	Let $\Psi$ be a root subsystem of $\Phi$. By considering $\Psi\cap\Phi_B$ as a root subsystem of $\Phi_B,$ we can define a $\bz$-linear function $p:\mathrm{Gr}(\Psi)\cap\mr{\Phi}_{s\ell}\to \Lambda_s$ such that $p_\alpha\in Y_\alpha(\Psi)$ for all $\alpha\in \mathrm{Gr}(\Psi)\cap\mr{\Phi}_{s\ell}.$ Hence $\Psi$ is of the form 
	\begin{equation}\label{eq2023d}
		\Psi=\{\alpha\oplus Y_\alpha(\Psi):\alpha\in\mr{\Phi}_d\}\cup \{\alpha\oplus(p_\alpha+Y_\ell'):\alpha\in \mr{\Phi}_{\ell}\}\cup\{\alpha\oplus(p_\alpha+Y_s'):\alpha\in \mr{\Phi}_{s}\}.
	\end{equation}
	
	\noindent \textit{We shall use this fact in this section without any further notice.} The next two Lemmas are the key Lemmas of this section. The first one determines when $\Phi_B$ is a maximal root subsystem of $\Phi.$ Recall from Definition \ref{extda} that $0$ may not belong to $\Lambda_d.$
	
	\begin{lem}\label{nonredBmax}
		Let $\Psi$ be a maximal root subsystem of $\Phi.$ Then we have that $\Psi\cap \Phi_B$ is a root subsystem of $\Phi_B.$ Moreover, $\Phi_B$ is a maximal root subsystem of $\Phi$ if and only if $\Lambda_d$ is equal to a coset in $\Lambda_d/2\Lambda_\ell.$ Otherwise, $\Phi_B$ is contained in a maximal root subsystem of the form $$\Psi':=\{\alpha\oplus Y_d:\alpha\in \mr{\Phi}_d\}\cup \Phi_B,$$ where $Y_d\subseteq \Lambda_d$ is a (fixed) maximal union of cosets in $\Lambda_d/2\Lambda_\ell$ for all $\alpha\in \mr{\Phi}_d.$
	\end{lem}
	
	\begin{proof}
		Note that we have $\Psi\cap\Phi_B\neq \emptyset.$ The proof of the fact that $\Psi\cap\Phi_B$ is a root subsystem of $\Phi$ is easy and we skip the details.  Recall that $\Lambda_\ell$ is a group. Define $\Psi':=\Phi_B\cup\{\alpha\oplus Z_d':\alpha\in \mr{\Phi}_d\},$ where $Z_d'\subseteq \Lambda_d$ is any (fixed) union of cosets in $\Lambda_d/2\Lambda_\ell.$ For $\alpha\in\mr{\Phi}_s,\beta\in\mr{\Phi}_\ell,$ the following relations show that $\Psi'$ is a root subsystem of $\Phi$ which contains $\Phi_B$
		$$Z_d'\pm 4\Lambda_\alpha\subseteq Z_d'+2\Lambda_\ell\subseteq Z_d',\ \ Z_d'\pm2\Lambda_\beta=Z_d'+2\Lambda_\ell\subseteq Z_d',\ \ Z_d'+2Z_d'\subseteq Z_d'+2\Lambda_\ell\subseteq Z_d'.$$

		Now the proof follows.
	\end{proof}
	\begin{rem}
		Lemma $6.2$ in \cite{kus2021borel} is not true if we drop the `closeness' of the root subsystem $\Psi.$
	\end{rem}
	
	The proof of the following Lemma is similar and hence we skip the details.
	
	\begin{lem}\label{nonredCmax}
		$\Phi_C$ is a maximal root subsystem of $\Phi$ if and only if $\Lambda_s=\Lambda_\ell.$ If $\Lambda_s\neq \Lambda_\ell,$ then $\Phi_C$ is contained in a maximal root subsystem of the form $$\Psi':=\{\alpha\oplus Y_\alpha:\alpha\in \mr{\Phi}_s\}\cup \Phi_C,$$ where $Y_\alpha\subseteq \Lambda_s$ is a maximal union of cosets in $\Lambda_s/\Lambda_\ell$ for all $\alpha\in \mr{\Phi}_s.$
	\end{lem}
	
	Now assume that $\Psi$ is a maximal root subsystem of $\Phi$ such that $\mathrm{Gr}(\Psi)=\mr{\Phi}$. Set $\Psi_B:=\Psi\cap\Phi_B.$ Then either $\Psi_B=\Phi_B$ or $\Psi_B=\Psi_B^{\mathrm{m}}$ for some maximal root subsystem of $\Phi_B$. If $\Psi_B\neq \Phi_B,$ the next Proposition determines $\Psi$ depending on the form of $\Psi_B^{\mathrm{m}}.$

	\begin{prop}\label{nonredMaxBproper}
		Let $\Psi$ be a maximal root subsystem of $\Phi$ such that $\mathrm{Gr}(\Psi)=\mr{\Phi}$ and $\Psi_B\neq \Phi_B.$ Assume that $\Psi_B^{\mathrm{m}}$ is the maximal root subsystem of $\Phi_B$ defined as above. If $\Psi_B^{\mathrm{m}}$ is defined by
		\begin{enumerate}
			\item Lemma \ref{lemredoftypeB}(\ref{lemlongfullb}),  then $\Psi=\Psi_B^{\mathrm{m}}\cup \{\alpha\oplus \Lambda_d:\alpha\in\mr{\Phi}_d\},$
			
			\item Lemma \ref{lemredoftypeB}(\ref{lemdetmb}), then $\Psi=\Psi_B^{\mathrm{m}}\cup \{\alpha\oplus (p_\alpha+H)\cap\Lambda_d:\alpha\in\mr{\Phi}_d\}$ for some $\bz$-linear function $p:\mr{\Phi}_{\ell d}\to \Lambda_\ell,\  p_\alpha\in Y_\alpha(\Psi),$ 
			
			\item Lemma \ref{lemredoftypeB}(\ref{lemdetnotmb}),
			\begin{enumerate}[leftmargin=*]
				\item and $\Lambda_d \subseteq H,$ then $\Psi=\Psi_B^{\mathrm{m}}\cup \{\alpha\oplus \Lambda_d:\alpha\in\mr{\Phi}_d\},$
				\item and $\Lambda_d \not\subseteq H,$ then there exists a $\bz$-linear function $p':\mr{\Phi}_{\ell d}\to \Lambda_\ell$ satisfying $p'_\beta=p_\beta$ for $\beta\in\{\epsilon_i-\epsilon_{i+1}:1\le i<n\}$ so that  
				$$\Psi=\Psi_B^{\mathrm{m}}\cup \{\alpha\oplus (p_\alpha'+Y_d'):\alpha\in\mr{\Phi}_d\},$$
				
				where $Y_d'=\cup_{b\in B} (b+2H)$ is a maximal union of cosets in $H/2H$ satisfying $a_n+A\subseteq A,\ \ 2a_n+B\subseteq B,$ if $S=\cup_{a\in A}(a+H)$ and $ a_n=2p_{\epsilon_n}-p'_{2\epsilon_n}.$
			\end{enumerate}
		\end{enumerate}
	\end{prop}
	\begin{rem}
		Note that the maximal root subsystems $(1)$ and $(2)$ defined above are of the form of Theorem $4.2$ and Theorem $4.3$ respectively in \cite{kus2021borel}. Also 
		note that $\Psi'$ defined in Lemma \ref{nonredCmax} is of the form Proposition \ref{nonredMaxBproper}(1).
	\end{rem}
	
	\begin{proof}
		We shall only prove $(3)(b)$ since the other cases are easy to prove. Let $\Psi$ be a maximal root subsystem such that $\Psi_B^{\mathrm{m}}$ is defined by Lemma \ref{lemredoftypeB}(\ref{lemdetnotmb}) and $\Lambda_d \not\subseteq H.$ Define a $\bz$-linear function $p':\mr{\Phi}_{\ell d}\to \Lambda_\ell$ by extending $p'_{\epsilon_i-\epsilon_{i+1}}=p_{\epsilon_i-\epsilon_{i+1}}$ for $1\le i<n$ and $p'_{2\epsilon_n}\in Y_{2\epsilon_n}$ arbitrary. Note that we have $2p_{\epsilon_i}-p_{2\epsilon_i}'=2p_{\epsilon_n}-p_{2\epsilon_n}'$ for all $i=1,2,\dots,n$. Now as in Proposition \ref{proprednotoftypeB}, $Y_d':=Y_\alpha-p_\alpha',\ \alpha\in\mr{\Phi}_d$ is a union of cosets in $H/2H$ so that if $Y_d'=\cup_{b\in B}(b+2H),$ then $0\in B$ and $B\subseteq H.$  Now the set $$\Psi':=\bigcup_{\substack{i, j=1\\ i\neq j}}^n\{\pm\epsilon_i\oplus (p_{\pm\epsilon_i}+S)\}\cup\{\pm\epsilon_i\pm\epsilon_j\oplus (p_{\pm\epsilon_i\pm\epsilon_j}+H)\}\cup\{\pm2\epsilon_i\oplus(p'_{\pm 2\epsilon_i}+Y_d')\}$$
		is a root subsystem if and only if for $\lambda,\mu\in\{\pm 1\}$ and $i=1,2,\dots,n$ the following hold:
		$$(p_{_{\lambda\epsilon_i}}+S)-\lambda\mu(p'_{2\mu\epsilon_i}+Y_d')\subseteq (p_{_{-\lambda\epsilon_i}}+S),\ \ (p'_{2\mu\epsilon_i}+Y_d')-4\lambda\mu(p_{_{\lambda\epsilon_i}}+S)\subseteq (p'_{-2\mu\epsilon_i}+Y_d'),$$ which is true if and only if $a_n+A\subseteq A,\ \ 2a_n+B\subseteq B.$ Therefore $\Psi$ is of the given form. This completes the proof.
	\end{proof}

	Let $\Psi$ be a maximal root subsystem of $\Phi.$ If $\mathrm{Gr}(\Psi)\neq \mathring{\Phi},$ then $\mathrm{Gr}(\Psi)$ is a maximal root subsystem of $\mr{\Phi}$ and thus of the form Lemma \ref{nonredgrmaximal}. Remark \ref{AJhatmax}, Lemma \ref{nonredBmax} and Lemma \ref{nonredCmax} describe precisely when $\widehat{\mathrm{Gr}(\Psi)}$ is a maximal root subsystem of $\Phi.$ If $\mathrm{Gr}(\Psi)=\mr{\Phi},$ then it is easy to see that $\Psi_B:=\Psi\cap\Phi_B$ is either $\Phi_B$ or maximal root subsystem of $\Phi_B.$ These cases have been discussed in Lemma \ref{nonredBmax}, Lemma \ref{nonredCmax} and Proposition \ref{nonredMaxBproper}. Summarizing, we have the classification Theorem of maximal root subsystems of a non-reduced irreducible affine reflection system. 
	\begin{thm}\label{thmnonred}
		Let $\Phi$ be the affine reflection system so that $\mr{\Phi}$ is a non-reduced irreducible finite root system. Let $\Psi$ be a root subsystem of $\Phi.$ Then $\Psi$ is a maximal root subsystem of $\Phi$ if and only if one of the following holds:
		\begin{enumerate}
			\item $\Psi=\widehat{A_J}$ for some non-empty proper subset $J$ of $I.$
			\item $\Lambda_\ell= \Lambda_s$ and $\Psi=\Phi_C=\widehat{A_\emptyset}.$
			\item $\Lambda_\ell\neq \Lambda_s$ and $\Psi=\Psi'$ as defined in Lemma \ref{nonredCmax}.
			\item $\Lambda_d$ is equal to a single coset in $\Lambda_d/2\Lambda_\ell$ and $\Psi=\widehat{B^I}.$
			\item $\Lambda_d$ is a union of more than one cosets in $\Lambda_d/2\Lambda_\ell$ and $\Psi=\Psi'$ as defined in Lemma \ref{nonredBmax}.
			\item $\Psi$ is given by Proposition \ref{nonredMaxBproper}.
		\end{enumerate}
	\end{thm}
	\begin{proof}
		Let $\Psi$ be a maximal root subsystem of $\Phi.$ Set $I_\Psi:=\{i\in I: \epsilon_i\in \mathrm{Gr}(\Psi)\}.$ If $\emptyset\neq I_\Psi\subsetneq I,$ then $\mathrm{Gr}(\Psi)=A_{I_\Psi}$ by Lemma \ref{nonredgrmaximal} and hence $\Psi=\widehat{A_{I_\Psi}}.$ Now assume that $I_\Psi=I.$ If $\mathrm{Gr}(\Psi)\subsetneq\mr{\Phi},$ then $\mathrm{Gr}(\Psi)=B^I$ by Lemma \ref{nonredgrmaximal} and thus $\Psi=\widehat{B^I}$ and $\Lambda_d= 2\Lambda_\ell$ by Lemma \ref{nonredBmax}. If $\mathrm{Gr}(\Psi)=\mr{\Phi},$ then set $\Psi_B:=\Psi\cap\Phi_B.$ If $\Psi_B=\Phi_B,$ then $\Lambda_d\neq 2\Lambda_\ell$ and $\Psi=\Psi'$ as  defined in Lemma \ref{nonredBmax}. Otherwise, $\Psi_B\subsetneq\Phi_B$ and thus $\Psi$ is given by Proposition \ref{nonredMaxBproper}. Lastly, we assume that $I_\Psi=\emptyset,$ then we must have $\Lambda_\ell=\Lambda_s$ and  $\Psi=\Phi_C$ by Lemma \ref{nonredCmax}.
		
		Conversely, each root subsystem defined above is a maximal root subsystem of $\Phi$ by Remark \ref{AJhatmax}, Lemma \ref{nonredBmax}, Lemma \ref{nonredCmax} and Proposition \ref{nonredMaxBproper}.
	\end{proof}
	\begin{rem}\label{nonredmaxclsd}
		If $\Psi$ defined above is a closed root subsystem, then $\Psi$ is of the form \cite[Theorem 4]{kus2021borel}. Therefore, we colclude that $\Psi$ is closed if and only if $\Psi$ is given by $(1),(2),(3)$ of the above Theorem and Proposition \ref{nonredMaxBproper}(1).
	\end{rem}

	\section{Applications}\label{sectionnullityle2}
	
	\subsection{Maximal root subsystems for nullity \texorpdfstring{$1$}{1}} It is well known that an affine reflection system is of nullity $1$ if and only if it is an affine root system (\cite{neher2011extended}). In (\cite{dyer2011reflection,dyer2011root}), the reflection subgroups of an untwisted affine root system have been classified. The maximal closed root subsystems of an affine root system were classified in \cite{roy2019maximal}. The characterization of maximal root subsystems for twisted affine root systems was open and we complete that here using our main results. At first we state the result for untwisted case in our language in the next Proposition. The proof follows easily from Theorem \ref{thmrednotoftypeB} and Theorem \ref{thmredoftypeB} by noting the fact that we have that $\Lambda_\ell=\Lambda_s$ and $m\Lambda_s$ is a maximal subgroup of $\Lambda_s.$
	\begin{prop}\label{propaffuntwisted}
		Let $\Psi$ be an untwisted affine root system. A subset $\Psi$ of $\Phi$ is a maximal root subsystem of $\Phi$ if and only if one of the following holds:
		\begin{enumerate}[leftmargin=*]
			\item $\mathrm{Gr}(\Psi)$ is a maximal root subsystem of $\mr{\Phi}$ such that $\mathrm{Gr}(\Psi)\neq \mr{\Phi}_s$ and $\Psi=\widehat{\mathrm{Gr}(\Psi)}.$ 
			\item\label{shortfullaffineuntwisted} there exists a $\bz$-linear function $p:\mr{\Phi}\to \Lambda_s,\ \ p_\alpha\in Y_\alpha,\ \ \forall \alpha\in \mr{\Phi}$
			\begin{enumerate}
				\item such that $\Psi$ is of the form $$\Psi=\{\alpha\oplus \Lambda_s:\alpha\in\mr{\Phi}_s\}\cup \{\alpha\oplus (p_\alpha+m\Lambda_s):\alpha\in\mr{\Phi}_\ell\},$$
				\item and a prime number $q\neq m$ such that $\Psi$ is of the form $$\Psi=\{\alpha\oplus (p_\alpha+q\bz):\alpha\in\mr{\Phi}\}.$$
			\end{enumerate}
		\end{enumerate}
	\end{prop}
	
	Now assume that $\Phi$ is a twisted affine reflection system of nullity $1$ i.e. a twisted affine root system. If $\mr{\Phi}$ is reduced, then we have $\Lambda_s=\bz$ and $\Lambda_\ell=m\bz.$ The following Proposition characterizes the maximal root subsystems of $\Phi$ where $\mr{\Phi}$ is reduced.
	\begin{prop}
		Let $\Phi$ be an affine root system such that $\mr{\Phi}$ is reduced. A root subsystem $\Psi$ of $\Phi$ is a maximal root subsystem if and only one of the following holds:
		\begin{enumerate}[leftmargin=*]
			\item $\mathrm{Gr}(\Psi)$ is a maximal root subsystem of $\mr{\Phi}$ such that $\mathrm{Gr}(\Psi)\neq \mr{\Phi}_\ell$ and $\Psi=\widehat{\mathrm{Gr}(\Psi)}.$
			\item $\mathrm{Gr}(\Psi)=\mr{\Phi}$ and if there exists a $\bz$-linear function $p:\mathrm{Gr}(\Psi)\to\bz,\  p_\alpha\in Y_\alpha(\Psi),\  \forall \alpha\in \mathrm{Gr}(\Psi)$ and of the following holds
			\begin{enumerate}
				\item $\Psi$ is of the form $$\Psi=\{\alpha+(p_\alpha+m\bz):\alpha\in\mr{\Phi}\},$$
				\item there exists a prime $q$ so that $\gcd(m,q)=1$ and $\Psi$ is of the form $$\Psi=\{\alpha+(p_\alpha+q\bz):\alpha\in\mr{\Phi}_s\}\cup\{\alpha+(p_\alpha+mq\bz):\alpha\in\mr{\Phi}_\ell\},$$
			\end{enumerate}
			
		\end{enumerate}
	\end{prop}
	\begin{proof}
		Proof follows from Proposition \ref{proprednotoftypeB} and Proposition \ref{propredoftypeB}.
	\end{proof}
	
	\subsection{Saito's EARS}\label{secnul2} An important class of extended affine root systems of nullity $2$ was introduced by Saito in \cite[Section 5]{saito1985extended}. We shall use Theorem \ref{thmrednotoftypeB}, Theorem \ref{thmredoftypeB} and Theorem  \ref{thmnonred} to characterize the maximal root subsystems of Saito's EARS we describe below.
	\begin{itemize}
		\item For a reduced irreducible finite root system $\mr{\Phi}$ and $a,b\in\bz$, the affine reflection system $\mr{\Phi}^{(a,b)}$ is given by $$\mathrm{Gr}(\mr{\Phi}^{(a,b)})=\mr{\Phi}\cup \{0\},\ \ \Lambda_s=\bz\times \bz,\ \ \Lambda_\ell=a\bz\times b\bz.$$\vspace{-10pt}
		\item The affine reflection system $B_n^{(2,2)*}$ is given by $$\mathrm{Gr}(B_n^{(2,2)*})=B_n\cup\{0\},\ \ \Lambda_s=\{(0,0),(1,0),(0,1)\}+2\bz\times 2\bz ,\ \ \Lambda_\ell=2\bz\times 2\bz.$$\vspace{-10pt}
		\item The affine reflection system $C_n^{(1,1)*}$ is given by $$\mathrm{Gr}(C_n^{(1,1)*})=C_n\cup\{0\},\ \ \Lambda_s=\bz\times \bz ,\ \ \Lambda_\ell=\{(0,0),(1,0),(0,1)\}+2\bz\times 2\bz.$$\vspace{-10pt}
		\item The remaining are non-reduced affine reflection systems. They are given by
		$$\mathrm{Gr}(BC_n^{(2,1)})=BC_n\cup\{0\} ,\ \ \Lambda_s=\bz\times \bz,\ \ \Lambda_\ell=\bz\times \bz,\ \ \Lambda_d= \{(1,0),(1,1)\}+2\bz\times 2\bz,$$
		$$\mathrm{Gr}(BC_n^{(2,4)})= BC_n\cup\{0\},\ \ \Lambda_s=\bz\times \bz,\ \ \Lambda_\ell=\bz\times 2\bz,\ \ \Lambda_d=\{(1,0)\}+2\bz\times 4\bz,$$
		$$\mathrm{Gr}(BC_n^{(2,2)}(1))= BC_n\cup\{0\},\ \ \Lambda_s=\bz\times \bz,\ \ \Lambda_\ell=\bz\times \bz,\ \ \Lambda_d= \{(1,0)\}+2\bz\times 2\bz,$$
		$$\mathrm{Gr}(BC_n^{(2,1)}(2))= BC_n\cup\{0\},\ \ \Lambda_s=\bz\times \bz,\ \ \Lambda_\ell=\bz\times 2\bz,\ \ \Lambda_d=\{(1,0)\}+2\bz\times 2\bz.$$
	\end{itemize}
	\vspace{-20pt}

	\subsubsection{}\label{rank2sp} Before classifying Saito’s EARS for rank $2$, we note that additional properties are present in rank $2$ compared to the general case. We provide detailed explanations of these additional properties below.
	\begin{enumerate}[leftmargin=*]
		\item If $\Psi$ is defined by Lemma \ref{lemrednotoftypeB}(\ref{lemshortfull}), $\Lambda_\ell$ is a subgroup and $m\Lambda_s\subsetneq\Lambda_\ell\subsetneq \Lambda_s$: Since $m\Lambda_s$ is a maximal subgroup of $\Lambda_\ell$ of index $2,$ it follows that
		$\Psi$ is of the form
		\begin{equation}\label{61r2}
			\Psi=\{\alpha\oplus \Lambda_s:\alpha\in\mr{\Phi}_s\}\cup\{\alpha\oplus (p_\alpha+m\Lambda_s):\alpha\in\mr{\Phi}_\ell\}
		\end{equation}
		Dually, if $\Phi$ is of type $C_I$ and $\Psi$ is defined by Lemma \ref{lemredoftypeB}(\ref{lemlongfullb}) or $\Phi$ is not of type $C_I$ and $\Psi$ is defined by Lemma \ref{lemrednotoftypeB}(\ref{lemdetm}), $\Lambda_s$ is a subgroup and $\Lambda_\ell\neq \Lambda_s,$ then $\Psi$ is of the form
		\begin{equation}\label{71r2}
			\Psi=\{\alpha\oplus(p_\alpha+ \Lambda_\ell):\alpha\in\mr{\Phi}_s\}\cup\{\alpha\oplus \Lambda_\ell:\alpha\in\mr{\Phi}_\ell\}
		\end{equation}
		
		\noindent If $\Lambda_\ell$ is not a subgroup satisfying $2\Lambda_s\subsetneq\Lambda_\ell\subsetneq \Lambda_s$, then $\Lambda_\ell$ is a union of at most three cosets in $\Lambda_s/2\Lambda_s.$ In particular, any maximal union of cosets of $m\Lambda_s$ in $\Lambda_\ell$ containing $0$ is a subgroup of $\Lambda_s.$ Therefore, there exists a subgroup $H'$ of $\Lambda_s$ satisfying $m\Lambda_s\subseteq H'\subseteq \Lambda_\ell$ such that 
		\begin{equation}\label{extra71}
			\Psi=\{\alpha\oplus\Lambda_s:\alpha\in\mr{\Phi}_s\}\cup\{\alpha\oplus (p_\alpha+H'):\alpha\in\mr{\Phi}_\ell\}
		\end{equation}
		\item Now we consider the case when $\Psi$ is defined by Lemma \ref{lemredoftypeB}(\ref{lemdetnotmb}). Recall the set $S$ and the fact that $\Lambda_\ell$ is a subgroup of $\langle \Lambda_s\rangle.$ We identify $\langle \Lambda_s\rangle$ with $\bz\times \bz.$ \medskip
		
		\begin{itemize}[leftmargin=*]
			\item \textbf{Case I:} Let $\Lambda_\ell=\Lambda_s.$ By Lemma \ref{lemmaximalsubgroup}, for any maximal subgroup $H$ of $\Lambda_\ell,$ we have $\det(H)/\det(\Lambda_\ell)=2$ if and only if $2\Lambda_s\subseteq H.$ Therefore, it follows that $\Psi$ is of the form \begin{equation}\label{73r2eq}
				\Psi=\{\alpha\oplus (p_\alpha+H):\alpha\in \mr{\Phi}\}
			\end{equation}
			for some maximal subgroup $H$ of $\Lambda_\ell$ such that $\det(H)/\det(\Lambda_\ell)=p\neq 2.$

			\item \textbf{Case II:} Now assume that $\Lambda_\ell=2\langle\Lambda_s\rangle.$ Then $\Lambda_\ell$ has a Hermite normal basis of the form $\left(\begin{smallmatrix} 2& 0\\ 0 &2 \end{smallmatrix}\right).$ Therefore, for some prime $p,$ the subgroup $H$ has a Hermite normal basis of exactly one of the following forms:
			$$\begin{pmatrix} 2p & 0 \\ 0 &2 \end{pmatrix},\ \text{ or }\ \begin{pmatrix}
				2 & 2x\\ 0 &2p
			\end{pmatrix}\ \ \text{for some }0\le x<p.$$
			
			If $H$ is given by the first form, the coset representatives $A$ of $H$ in $\langle \Lambda_s \rangle$ is given by $A=\{(2r,0),(2r+1,0),(2r,1),(2r+1,1):0\le r<p\}.$ An easy computations of when $2a\in H$ for $a\in A$ implies that 
			\begin{equation}\label{funnyeq1}
				S\subseteq\begin{cases}
					\bigcup_{r=0}^{p-1} ((2r,0)+H)\cup ((2r,1)+H) & \text{ if }p=2,\\
					H\cup ((p,0)+H)\cup((p,1)+H)\cup ((0,1)+H) & \text{ if }p\neq 2.
				\end{cases}
			\end{equation}

			Now assume that $H$ is given by the second form. The coset representatives of $H$ in $\langle \Lambda_s \rangle$ is given by $A=\{(0,2r),(0,2r+1),(1,2r),(1,2r+1):0\le r<p\}.$ A simple computation implies \begin{itemize}
				\item $(0,4r)\in H$ if and only if either $p=2$ or $p\neq 2$ and $r=0,$
				\item $(0,4r+2)\in H$ if and only if $p\neq 2$ and $r=(p-1)/2,$
				\item $(2,4r)\in H$ if and only if $r=x/2$ or $r=(x+p)/2,$
				\item $(2,4r+2)\in H$ if and only if $r=(x-1)/2$ or $r=(x+p-1)/2.$
			\end{itemize}
			Therefore, we obtain 
			\begin{equation}\label{funnyeq2}
				S\subseteq\begin{cases}
					\Lambda_\ell\cup ((1,0)+H)\cup ((1,2)+H)& \text{ if } p=2 \text{ and }x=0,\\
					\Lambda_\ell\cup ((1,1)+H)\cup ((1,3)+H)& \text{ if } p=2 \text{ and }x=1,\\
					H\cup ((1,x)+H)\cup ((1,x+p)+H)\cup ((0,p)+H)& \text{ if } p\neq 2.
				\end{cases}
			\end{equation}
			
			If we denote by $S'$ the right side of Equation \eqref{funnyeq1} and \eqref{funnyeq2}, then $\Psi$ is of the form 
			\begin{equation}\label{72r2IIInotgroup}
				\Psi=\{\alpha\oplus (p_\alpha+S')\cap \Lambda_s:\alpha\in\mr{\Phi}_s\}\cup \{\alpha\oplus (p_\alpha+H):\alpha\in\mr{\Phi}_\ell\},
			\end{equation}
			
			If $\Lambda_s$ is a group, then in both the cases, $S=S'$ and $S$ is a subgroup of $\Lambda_s.$ More precisely, $\Psi$ is of the form 
			\begin{equation}\label{72r2IIIgroup}
				\Psi=\{\alpha\oplus (p_\alpha+H'):\alpha\in\mr{\Phi}_s\}\cup \{\alpha\oplus (p_\alpha+H):\alpha\in\mr{\Phi}_\ell\},
			\end{equation}
			where $H'$ is the subgroup of $\Lambda_s$ given by the Hermite normal basis $B'$ where 
			$$B'=\begin{cases}
				\left(\begin{smallmatrix}p & 0 \\0 & 1 \end{smallmatrix}\right) & \text{ if } H \text{ is of the first form},\\
				\left(\begin{smallmatrix}1 & x \\0 & p \end{smallmatrix}\right) & \text{ if } H \text{ is of the second form}.
			\end{cases}$$
			
			\item \textbf{Case III:} Assume that $\langle2\Lambda_s\rangle\subsetneq \Lambda_\ell\subsetneq  \Lambda_s.$ Note that $\Lambda_s$ must be a subgroup in this case. Similar computation as above shows that there exists a subgroup $H'$ of $\Lambda_s$ which contains $H$ so that $\Psi$ is of the form 
			\begin{equation}\label{73r2remaining}
				\Psi=\{\alpha\oplus (p_\alpha+H'):\alpha\in\mathring{\Phi}_s\}\cup\{\alpha\oplus (p_\alpha+H):\alpha\in \mathring{\Phi}_s\}.
			\end{equation}
			More precisely, the Hermite normal basis $B'$ of $H'$ is given by Table \ref{FunnyTable} ($0\le x<1,\ 0\le y<p$).
			
			\noindent In any of the above three cases, for $m=2,$ $\Psi^\vee$ is given by Lemma \ref{lemrednotoftypeB}(\ref{lemdetnotm}) where $H$ is explicitly given by above.
		\end{itemize}
	\end{enumerate}
	
	\begin{center}
		\begin{table}
			{\renewcommand{\arraystretch}{1.25}
				\renewcommand{\tabcolsep}{10pt}
				\begin{NiceTabular}{|c|c|c|c|c|}[hvlines]
					$\Lambda_\ell$ &  $H$ &  $p$&  Sub cases &  $B'$\\
					\Block{4-1}{$\begin{pmatrix}2&0\\0&1\end{pmatrix}$} & $\begin{pmatrix} 2p& 0\\0&1\end{pmatrix}$ & Any & Any & $\begin{pmatrix}p&0\\0&1 \end{pmatrix}$\\
					& \Block{3-1}{$\begin{pmatrix}2&y\\0&p\end{pmatrix}$} & $p=2$  & $y=1$  & $\begin{pmatrix}2&0\\0&1\end{pmatrix}$ \\
					& & \Block{2-1}{$p\neq 2$} & $y$-odd & $\begin{pmatrix}1&\frac{y+p}{2}\\0&p\end{pmatrix}$ \\
					& & & $y$-even & $\begin{pmatrix}1&\frac{y}{2}\\0&p\end{pmatrix}$ \\
					\Block{3-1}{$\begin{pmatrix}1&x\\0&2\end{pmatrix}$} & $\begin{pmatrix} p& px\\0&2\end{pmatrix}$ & $p\neq 2$ & Any & $\begin{pmatrix}p&0\\0&1 \end{pmatrix}$\\
					& \Block{2-1}{$\begin{pmatrix}1&x+2y\\0&2p\end{pmatrix}$} & $p=2$ & \Block{2-1}{Any} & $\begin{pmatrix}1&x\\0&2\end{pmatrix}$\\
					& & $p\neq 2$  &  & $\begin{pmatrix}1&x+2y(\mathrm{mod}\ p)\\0&p\end{pmatrix}$
			\end{NiceTabular}}
			\caption{The case $2\Lambda_s\subsetneq \Lambda_\ell\subsetneq  \Lambda_s.$}
			\label{FunnyTable}
		\end{table}
	\end{center}
	
	\subsubsection{Classification}
	We are now in a position to classify the maximal root subsystems of Saito's EARS. Table \ref{Tablenonreduced} describes the maximal root subsystems of Saito's EARS with full gradient. For non-reduced case, the table can be read with the forms obtained in Section \ref{rank2sp} for rank $2.$
	
	\begin{center}
		\begin{table}[ht]
			{\renewcommand{\arraystretch}{1.5}
				\renewcommand{\tabcolsep}{4pt}
				\begin{NiceTabular}{|c|c|c|c|c|}[hvlines]
					Class & $\Lambda$  & Form  & Gradient type &  Maximal root subsystems  \\
					\Block{14-1}{\text{Reduced}} & \Block{12-1}{$\begin{array}{c} \Lambda_s \text{ and } \\
							\Lambda_\ell\\
							\text{are both}\\
							\text{groups}
						\end{array}$} & \Block{3-1}{$\mr{\Phi}^{(1,1)}$}   & $B$  & Lemma \ref{lemredoftypeB}(\ref{lemdetmb}), Eq \eqref{73r2eq} \\
					& & & $C$ & Lemma \ref{lemrednotoftypeB}(\ref{lemshortfull}), $\Psi^\vee$ ($\Psi$ in Eq \eqref{72r2IIIgroup}) \\
					& & & $\not= B,C$ & Lemma \ref{lemrednotoftypeB}(\ref{lemshortfull}), Lemma \ref{lemrednotoftypeB}(\ref{lemdetnotm})  \\
					&	& \Block{3-1}{$\mr{\Phi}^{(1,m)}$}   & $B$  & Eq \eqref{71r2}, Lemma \ref{lemredoftypeB}(\ref{lemdetmb}), Eq \eqref{73r2remaining}\\
					& & & $C$ & Eq \eqref{61r2}, Lemma \ref{lemrednotoftypeB}(\ref{lemdetm}), $\Psi^\vee$ ($\Psi$ in Eq \eqref{73r2remaining}) \\
					& & & $\not= B,C$ & Lemma \ref{lemrednotoftypeB}(\ref{lemshortfull}), Lemma \ref{lemrednotoftypeB}(\ref{lemdetm}), Lemma \ref{lemrednotoftypeB}(\ref{lemdetnotm}) \\
					&	& \Block{3-1}{$\mr{\Phi}^{(m,1)}$}   & $B$  &Eq \eqref{71r2}, Lemma \ref{lemredoftypeB}(\ref{lemdetmb}), Eq \eqref{73r2remaining}\\
					& & & $C$ & Eq \eqref{61r2}, Lemma \ref{lemrednotoftypeB}(\ref{lemdetm}), $\Psi^\vee$ ($\Psi$ in Eq \eqref{73r2remaining})\\
					& & & $\not= B,C$ & Lemma \ref{lemrednotoftypeB}(\ref{lemshortfull}), Lemma \ref{lemrednotoftypeB}(\ref{lemdetm}), Lemma \ref{lemrednotoftypeB}(\ref{lemdetnotm}) \\
					&	& \Block{3-1}{$\mr{\Phi}^{(m,m)}$}   & $B$ & Lemma \ref{lemredoftypeB}(\ref{lemlongfullb}), Eq \eqref{72r2IIIgroup}\\
					& & & $C$ & Lemma \ref{lemrednotoftypeB}(\ref{lemdetm}), $\Psi^\vee$ ($\Psi$ in Eq \eqref{73r2eq})\\
					& & & $\not= B,C$ & Lemma \ref{lemrednotoftypeB}(\ref{lemdetm}), Lemma \ref{lemrednotoftypeB}(\ref{lemdetnotm})   \\
					&	\Block{2-1}{$\begin{array}{c}
							\text{One of }\Lambda_s \\
							\text{ and } \Lambda_\ell \text{ is}\\
							\text{not a group}
						\end{array}$}& $B_n^{(2,2)*}$  &   & Lemma  \ref{lemredoftypeB}(\ref{lemlongfullb}), $\Psi$ in Eq \eqref{72r2IIInotgroup}\\
					&	& $C_n^{(1,1)*}$ & & Eq (\ref{extra71}), $\Psi^\vee$ ($\Psi$ in Eq \eqref{72r2IIInotgroup})  \\
					\Block{4-1}{Non-reduced} & \Block{4-1}{}  & $BC_n^{(2,1)}$  &\Block{4-1}{}  & Theorem \ref{thmnonred}(1,2,5), Proposition \ref{nonredMaxBproper}(2,3) \\
					&	&$BC_n^{(2,4)}$   &   & Theorem \ref{thmnonred}(1,3,4), Proposition \ref{nonredMaxBproper}(1,2,3)\\
					&	& $BC_n^{(2,2)}(1)$   &  & Theorem \ref{thmnonred}(1,2,4), Proposition \ref{nonredMaxBproper}(2,3)\\
					&	& $BC_n^{(2,2)}(2)$  &   & Theorem \ref{thmnonred}(1,3,5), Proposition \ref{nonredMaxBproper}(1,2,3)\\
			\end{NiceTabular}}
			\caption{Maximal root subsystems of Saito's EARS}
			\label{Tablenonreduced}
		\end{table}
	\end{center}

	\section*{Acknowledgement} The author thanks Prof. R. Venkatesh for many useful discussions. The author also thanks Prof. Apoorva Khare for providing the opportunity to work under his mentorship as a research associate, with partial support from the Swarna Jayanti Fellowship grant SB/SJF/2019-20/14, during which, part of this work was carried out. The author also acknowledges The Institute of Mathematical Sciences, Chennai, for their financial support and exceptional facilities that enabled the completion of this work.
	
	\bibliographystyle{plain}
	\bibliography{bibliography}

\end{document}